\documentclass{article}

\usepackage{ amssymb, amsmath, amsthm, amsfonts,mathrsfs, verbatim }
\usepackage{pgfplots}
\usepackage{tikz}
\usetikzlibrary{decorations.fractals, arrows, calc}

\usepackage{mathtools}

\usepackage{xcolor}
\usepackage{comment}
\usepackage{xy}
\xyoption{all}
\usepackage{cite}
\usepackage{cancel}
\usepackage{graphicx}
\usepackage{float}
\graphicspath{ {images/} }
\usepackage{fullpage}
\usepackage{ stmaryrd }
\usepackage{enumerate}
\usepackage{hyperref}%

\theoremstyle{definition}
\newtheorem{theorem}{Theorem}[section]

\newtheorem{proposition}[theorem]{Proposition}
\newtheorem{lemma}[theorem]{Lemma}
\newtheorem{example}[theorem]{Example}
\newtheorem{remark}[theorem]{Remark}
\newtheorem{definition}[theorem]{Definition}
\newtheorem{conjecture}[theorem]{Conjecture}
\newtheorem{notation}[theorem]{Notation}

\newcommand{\uu}[1]{\underline{x}}

\newcommand{\QQ}{\mathbb{Q}}

\newcommand{\NN}{\mathbb{N}}

\newcommand{\RR}{\mathbb{R}}

\newcommand{\UU}{\mathcal{U}}

\newcommand{\SSS}{\mathbb{S}}
\newcommand{\N}{\mathcal{N}}

\newcommand{\VR}{\mathrm{VR}}
\newcommand{\card}{\mathrm{card}}

\newcommand{\Alpha}{\alpha}

\newcommand{\Mag}{\mathrm{Mag}}
\newcommand{\Mink}{\mathrm{Mink}}

\newcommand{\define}[1]{{\bf \boldmath{#1}}}

\begin{document}

\title{Alpha magnitude}

\author{
Miguel O'Malley\thanks{
Wesleyan University,
\texttt{momalley@wesleyan.edu}}, 
 Sara Kalisnik\thanks{
ETH Zurich,
\texttt{sara.kalisnik@math.ethz.ch}},   and 
Nina Otter\thanks{
Queen Mary University of London,
\texttt{n.otter@qmul.ac.uk}}
}

\date{}


\maketitle

\begin{abstract} Magnitude is an isometric invariant for metric spaces that was introduced by Leinster around 2010, and is currently the object of intense research, since it has been shown to encode many known invariants of metric spaces. In recent work, Govc and Hepworth introduced persistent magnitude, a numerical invariant of a filtered simplicial complex associated to a metric space. Inspired by Govc and Hepworth's definition, we introduce alpha magnitude and investigate some of its key properties. Heuristic observations lead us to conjecture a relationship with the Minkowski dimension of compact subspaces of Euclidean space. Finally, alpha magnitude presents computational advantages over both magnitude as well as Rips magnitude, and we thus propose it as a new measure for the estimation of fractal dimensions of real-world data sets that is easily computable.
\end{abstract}

\section{Introduction}

Magnitude was introduced by Tom Leinster  \cite{leinster2010} in his study of cardinality-like invariants. A special case of the notion introduced by Leinster is the magnitude of a metric space, which is an isometric invariant of the space. The study of magnitude is currently a thriving research area;   research directions include the study of the relationship of magnitude with many known invariants from integral geometry and geometric measure theory (see, e.g., the survey \cite{LM17}), fractal dimensions \cite{Meckes2015}, magnitude homology \cite{LS17,HW17} and its relationship with persistent homology \cite{O18}, the study of biodiversity  \cite{LC12,LM16,leinster21}, information theory \cite{LR21,leinster21}, and boundary-detection problems in machine learning \cite{bunch2020weighting}.

More recently,  Govc and Heptworth \cite{GH21}  introduced  persistent magnitude, an invariant of a metric space that depends on a choice of a filtered simplicial complex  associated to the metric space. Given a finite metric space $(X,d)$, and a filtered simplicial complex $K=\{K_i\}_{i\in \mathbb{N}}$ associated to it that satisfies certain finiteness assumptions,  Govc and Hepworth define the {\bf persistent magnitude} of $X$ (with respect to $K$) to be the positive real number
\[
\sum_{k=0}^\infty \sum_{i=1}^{m_k}(-1)^k(e^{-a_{k,i}}-e^{-b_{k,i}}), 
\]
where $\{[a_{k,i},b_{k,i})\}_{i=1}^{m_k}$ is the barcode in degree $k$ of $K$. 
Govc and Hepworth then proceed to study properties of the persistent magnitude associated to the filtered Vietoris--Rips complex of a metric space, which they call {\bf Rips magnitude}. They observe that the infinite sum reduces to a finite sum:
\[
\sum_{k=0}^\infty \sum_{i=1}^{m_k}(-1)^k(e^{-a_{k,i}}-e^{-b_{k,i}})\, 
= \sum_{k=0}^{\card(X)} \sum_{i=1}^{m_k}(-1)^k(e^{-a_{k,i}}-e^{-b_{k,i}})\, ,\]
and that similarly to magnitude, Rips magnitude gives a measure of the effective number of points of a metric space. Furthermore, they propose extensions of the definition to compact metric spaces, and study what value  these extensions take for several compact subsets of Euclidean space.

Motivated by the application of persistent magnitude to the study of point-cloud data, we note that, taking into consideration current state-of-the-art software implementations, Rips magnitude   can be computed in practice only for point clouds of fewer than 35 points.  When $\card(X)=35$, we have that in the worst case, the filtered Vietoris--Rips complex contains $2^{35}\approx 34\times 10^9$ many simplices, which is a size of simplicial complex that exceeds the size allowed by state-of-the-art software. Here we are using as reference  the results from the benchmarking in \cite{roadmap}\footnote{We note that even if current advances might result in results better than the ones reported in  \cite{roadmap}, this would not make a difference in the unfeasibility of the computation of the Rips magnitude for point-cloud data sets with software packages.}.

 We therefore propose {\bf alpha magnitude}, which is the persistent magnitude associated to the filtered alpha complex of a subspace of Euclidean space. We show that alpha magnitude presents several advantages over both Rips magnitude and magnitude. Firstly, alpha magnitude has a lower computational complexity than   Rips magnitude, as discussed in the previous paragraph, as well as magnitude, see the discussion at the end of Section \ref{SS:mag dim}. 
 Secondly, we show that alpha magnitude is defined for certain compact subspaces of Euclidean space for which Rips magnitude is not defined, see Proposition \ref{circlealpha}. 
 Finally, we proceed to define {\bf alpha magnitude dimension}, similarly to how Meckes defines {\bf magnitude dimension} in \cite{Meckes2015}. In \cite{Meckes2015} Meckes showed that for a large class of compact subsets of Euclidean space for which magnitude is defined, magnitude dimension is equivalent to Minkowski dimension. Based on several  computations (see examples in Section \ref{SS:alpha mag dim}) as well as heuristic observations (see the estimates  in Section \ref{S:experiments}), we conjecture that alpha magnitude dimension is equivalent to the Minkowski dimension of compact subspaces of Euclidean space for which it is defined.  This motivates using alpha magnitude to estimate the Minkowski dimension of spaces for which it would be difficult to do so theoretically, as well as for real-world datasets, which is a direction that we plan to pursue in future work.

\section{Preliminaries}
We introduce the main notions related to magnitude, its relationship to persistent homology, as well as fractal dimensions.

\begin{notation}
We make note of some notation conventions we will make use of. Bold face letters (i.e, $\mathbf{w},\mathbf{u} $) represent vectors. We represent the components of each bold face vector as non-bold letters with subscript (i.e, $\mathbf{w}^T=(w_1,w_2,\dots,w_m)$). By $\mathbf{u}$ we denote the standard unitary vector, that is, the $m\times 1$ vector with all entries equal to $1$. Matrices are either represented with capital letters or $\zeta_A$ for the similarity matrix for the space $A$. A matrix with one subscript $i$ represent the $i$th column of the matrix, that is, $D_i$ represents the $i$th column of $D$. The term $D_{i,j}$ represents the $(i,j)$th entry of $D$.
\end{notation}

\subsection{Magnitude}

In this section, as well as in the rest of the paper, unless otherwise specified, we consider the standard Euclidean metric on $\mathbb{R}^n$. First we review some basic definitions and standard results on the properties of distance matrices.

\begin{definition}
    Let $(X,d)$ denote a finite metric space $X$ with a distance function $d\colon X\to\RR$. We define the {\bf distance matrix} $D$ of $(X,d)$ to be the $\mathbf{card}(X)\times \mathbf{card}(X)$ matrix with entries $D_{i,j}=d(x_i,x_j)$ for $x_i,x_j\in X$. 
\end{definition}

\begin{definition}
    Let $(X,d)$ be a finite metric space and $D$ its distance matrix. The \textbf{similarity matrix} $\zeta_X$ of $(X,d)$ is defined to be the matrix with entries $(\zeta_X)_{i,j}=e^{-D_{i,j}}$.
\end{definition}
\begin{definition}
    Let $(X,d)$, $\zeta_X$ be as above. Suppose $\textbf{w}$ exists such that $\zeta_X\textbf{w}=\textbf{u}$. Then we call $\textbf{w}$ a \textbf{weighting} on $X$.
\end{definition}

\begin{definition}
    Suppose there exists a weighting $\textbf{w}$ on the space $X$. The {\bf magnitude} of $X$ is defined as $$|X|=\displaystyle \sum_{i=1}^{\mathbf{card}(X)} w_i. $$
\end{definition}

\begin{example}
	The magnitude of the two point space $(X,d_X)$, where $X=\{x_1,x_2\}$ and ${d_X(x_1,x_2)=d}$ is $\dfrac{2}{1+e^{-d}}$. We can see this easily, since the similarity matrix $$\zeta_X=\begin{bmatrix} 1& e^{-d}\\ e^{-d} & 1\end{bmatrix}$$ has a weighting $$\textbf{w}=\begin{bmatrix}\dfrac{1}{1+e^{-d}}\\ \dfrac{1}{1+e^{-d}} \end{bmatrix}. $$ In fact, for any space where the isometry group acts transitively, the magnitude is the cardinality of the space multiplied by the sum of the entries along any row of the similarity matrix. These spaces are referred to as \textbf{homogeneous}. 
	\end{example}

\begin{definition}
Let $(X, d)$ be a metric space and $t\in (0, \infty)$. Then $tX$ denotes
the metric space with the same points as $X$ and the metric $d_t(x, y) = td(x, y)$ for $x, y \in X$.
\end{definition}

\begin{definition}
    The \textbf{magnitude function} of a finite metric space $(X,d)$ is the partially-defined function $t\mapsto |tX|$ defined for all $t\in (0, \infty)$ such that $tX$ has magnitude.  
\end{definition}

It is known that magnitude is defined for all metric spaces with four or fewer points. Below is an example of a metric space with $5$ points, due to Leinster \cite[Example 2.2.7]{leinster2010}, for which the magnitude function is not defined for all values of $t$:

\begin{example} 
	 Let $K_{3,2}$ be the bipartite complete graph on $3$ and $2$ vertices, and endow its set of vertices with the shortest-path metric. Then the similarity matrix of $tK_{3,2}$ is $$t\zeta_{K_{3,2}}=\begin{bmatrix}1 & e^{-2t} & e^{-2t} & e^{-t} & e^{-t} \\ e^{-2t} & 1 & e^{-2t} & e^{-t} & e^{-t}\\  e^{-2t} & e^{-2t} & 1 & e^{-t} & e^{-t}\\ e^{-t} & e^{-t} & e^{-t} & 1 & e^{-2t}\\ e^{-t} & e^{-t} & e^{-t} & e^{-2t} & 1 \end{bmatrix}.$$ Computed symbolically, the magnitude function of this space is $\dfrac{5-7e^{-t}}{(1+e^{-t})(1-2e^{-2t})}\, .$ The magnitude function is discontinuous at $\log(\sqrt{2})$. We note that this space is ill-behaved in other ways as well; the magnitude function takes on values that are larger than the cardinality of the space,  and it can be negative and decreasing, see the aforementioned reference for details. 
	\end{example}

\subsection{Persistent homology}

Persistent homology is a method used to study qualitative features of data using topology. Here we are interested in studying data in the form of points clouds, i.e., finite metric spaces that arise from applications. Point clouds, being discrete topological spaces,  have no non-trivial topological features. Thus, the idea is to assign for every parameter $\epsilon >0$ a topological space, more specifically, a simplicial complex, to the point cloud and then track the evolution of the topological features as the parameter varies. In this paper we work with alpha complexes, but there are other families of complexes one can assign, for example, Čech and Vietoris-Rips complexes.

\begin{definition}
	An \textbf{abstract simplicial complex} $(\Sigma,V)$ is given by a set  $V$  whose elements we call \define{vertices} and a set   $\Sigma$ of non-empty finite subsets of $V$. This data satisfies the following properties: we have that (i) $\{v\}\in \Sigma$ for all $v\in V$ and (ii) if $\sigma\in \Sigma$ and $\tau\subset\sigma$, then $\tau\in\Sigma$. If $\sigma\in \Sigma$ has cardinality $p+1$, we say that $\sigma$ is a \define{$p$-simplex}, or a simplex of \define{dimension $p$}. A \define{simplex} is a $p$-simplex for some $p\in \mathbb{N}$.
	\end{definition}
	
	Alpha and \v{C}ech complexes are related to  a general construction in algebraic topology, called a `nerve':

\begin{definition}
	The \textbf{nerve}  of a cover $\UU=\{U_i\}_{i\in I}$ of a topological space $S$ is the abstract simplicial complex $N(\UU)$ with the vertex set given by $I$ such that 
	$\sigma=\{i_{1},\dots,i_{n}\}$ is a simplex in $\N(\UU)$ if and only if   $ U_{i_1}\cap \dots \cap U_{i_n}\neq\emptyset$.
	\end{definition}

Under suitable assumptions on the sets in the cover, one has that the geometric realisation  $|N(\UU)|$ of the nerve $N(\UU)$ has the same homotopy type as the space $S$. This result is known as `nerve lemma', and there exist many different versions of it in the literature. For the purposes of our work, here it suffices to consider the statement of the Nerve Lemma as given in \cite[III.2]{EH10}: if $S\subset \mathbb{R}^d$ and if the sets in the cover of $S$ are convex and closed, then  $|N(\UU)|\simeq S$.

	\begin{definition} Let $(X, d)$ be a finite metric subspace of Euclidean space $\mathbb{R}^d$. Let ${\UU_\epsilon=\{B_\epsilon(x)\colon x\in X\}}$ be the cover of $X$ given by closed balls of radius $\epsilon$ centered at the points of $X$. 
	The \textbf{\v{C}ech complex at scale $\epsilon$}  is the abstract simplicial complex $\check{C}(X)_\epsilon=N(\UU_\epsilon)$. 
	\end{definition}
	
	Thus, by the nerve lemma, we have that $\cup_{x\in X}B_\epsilon(x)\simeq |\check{C}(X)_\epsilon|$, see \cite[III.2]{EH10}. The \v{C}ech complex has the drawback that it can have simplices in dimensions higher than the embedding space. One therefore often considers instead the alpha complex, for which the dimension is at most the dimension of the ambient space, provided the points are in general position \cite[III.3]{EH10}.

	\begin{definition} Let $(X, d)$ be a finite metric subspace of Euclidean space $\mathbb{R}^d$. 
	The \textbf{alpha complex of $X$ at scale $\epsilon$}  is the abstract simplicial complex $\Alpha(X)_\epsilon$ given by  the nerve of the collection of sets $\{U_x\}_{x\in X}$, where
	\[
	U_x=B_\epsilon(x)\cap V_x 
	\] 
	and  
	\[
	V_x=\{p\in \mathbb{R}^d \colon d(p,x)\leq d(p,x')\,\,\forall x'\in X \}
	\]
	is the \define{Voronoi cell} around $x$.	\end{definition}

Again by the nerve lemma, we have that  the geometric realisation of the alpha complex is homotopy equivalent to $\cup_{x\in X} B_\epsilon(x)$, see \cite[III.4]{EH10}.

Another type of simplicial complex that is widely used in applications is the Vietoris--Rips complex. Here we recall its defintion, since this is also the complex considered by Govc and Hepworth in their work on persistent magnitude \cite{GH21}. 

\begin{definition}
Given a point cloud $X$ and a real number $\epsilon \geq 0$, we define the \define{Vietoris-Rips complex of $X$ and $\epsilon$} to be:
 \[
\textrm{VR}(X)_\epsilon = \{\sigma \subseteq X \,|\, d(x,y)\leq \epsilon,\forall x,y \in \sigma\}.
 \]

\end{definition} 
	
\noindent 	
In other words, $\textrm{VR}(X)_\epsilon$ consists of all subsets of $X$ whose diameter is no greater than $\epsilon$.

All the simplicial complexes introduced in this section yield filtered simplicial complexes: a \define{filtered simplicial complex} is a collection $K=\{K_\epsilon\}_{\epsilon\in \mathbb{R}_{\geq 0}}$ of simplicial complexes indexed by non-negative real numbers with the property that $K_\epsilon\subset K_{\epsilon'}$ whenever $\epsilon\leq \epsilon'$.

Applying the homology functor $H_k$ in degree $k$  to a filtered simplicial complex, we obtain what is called a `persistence module'.

\begin{definition}
	A \textbf{persistence module} $\mathbf{V}$ is a collection of indexed vector spaces $\{ V_t|\,t\in\RR \}$ and linear maps $\{ v_b^a|\,v_b^a\colon V_a\to V_b,\,a\leq b \}$ such that the composition has the properties $v_c^b\circ v_b^a=v_c^a$ whenever $a\leq b\leq c$ and $v^a_b$ is the identity map whenever $a=b$. 
\end{definition}

If the homology $H_k$ is taken with coefficients in a field, and we consider a filtered simplicial complex $K$ associated to a finite point cloud, then the isomorphism class of the persistence module $H_k(K)$ can be classified by a finite collection of intervals, called a \define{barcode}.  We can interpret each interval in the barcode as describing the lifetime of a topological feature which appears at the value of a parameter given by the left-hand endpoint of the interval and disappears at the value given by the right-hand endpoint. 

\begin{example}
	Let $X\subset\RR^2$ such that $X=\{ (0,0),(1,0),(1/2,\sqrt{3}/2)  \}$. Below we show the alpha complex $\alpha(X)_\epsilon$ for different values of epsilon. 
	
	\begin{figure}[h!]
		\centering
		\begin{tikzpicture}[scale=2]
		\coordinate (A) at (-3,0) ;
		\coordinate (B) at (-2,0);
		\coordinate (C) at (-2.5,{sqrt(3)/2});
		\coordinate (Z) at (-2.5,-.5);
		\node at (A) [below left] {(0,0)};
		\node at (A) {$\bullet$};
		\node at (B) [below right] {(1,0)};
		\node at (B) {$\bullet$};
		\node at (C) [above] {(1/2,{$\sqrt{3}$/2})};
		\node at (C) {$\bullet$};
		\node at (Z) [below] {$\epsilon=0$};
		\filldraw[opacity=0, blue] (A)  -- (B) -- (C) -- cycle;
		\coordinate (D) at (-.5,0) ;
		\coordinate (E) at (.5,0);
		\coordinate (F) at (0,{sqrt(3)/2});
		\coordinate (Y) at (0,-.5);
		\node at (D) [below left] {(0,0)};
		\node at (D) {$\bullet$};
		\node at (E) [below right] {(1,0)};
		\node at (E) {$\bullet$};
		\node at (F) [above] {(1/2,{$\sqrt{3}$/2})};
		\node at (F) {$\bullet$};
		\draw[black, very thick] (D)--(E);
		\draw[black, very thick] (E)--(F);
		\draw[black, very thick] (F)--(D);
		\node at (Y) [below] {$\epsilon=1/2$};
		\filldraw[opacity=0, blue] (A)  -- (B) -- (C) -- cycle;
		\coordinate (G) at (2,0) ;
		\coordinate (H) at (3,0);
		\coordinate (I) at (2.5,{sqrt(3)/2});
		\coordinate (X) at (2.5,-.5);
		\node at (G) [below left] {(0,0)};
		\node at (G) {$\bullet$};
		\node at (H) [below right] {(1,0)};
		\node at (H) {$\bullet$};
		\node at (I) [above] {(1/2,{$\sqrt{3}$/2})};
		\node at (I) {$\bullet$};
		\draw[black, very thick] (G)--(H);
		\draw[black, very thick] (H)--(I);
		\draw[black, very thick] (I)--(G);
		\node at (X) [below] {$\epsilon=1/\sqrt{3}$};
		\filldraw[opacity=0.5, cyan] (G)  -- (H) -- (I) -- cycle;
		\end{tikzpicture}
	\end{figure}
	Then the corresponding barcode will be:
	\begin{figure}[h!]
		\centering
		\begin{tikzpicture}[scale=2]
		\draw[dashed, lightgray] (0,0) -- (0,.8);
		\draw[dashed, lightgray] (2,0) -- (2,.8);
		\draw[dashed, lightgray] ({4/sqrt(3)},0) -- ({4/sqrt(3)},.8);
		\draw[dashed, lightgray] (4,0) -- (4,.8);
		\draw [-latex,shorten >=-3pt] (-.2,0) -- (4,0) node [below right] {$\phantom{\sqrt{3}}\epsilon\phantom{\sqrt{3}}$};
		\foreach \x in  {0,2,4}
		\node at (\x,0) {$\bullet$};
		\node at (0,0) [below] {0};
		\node at (2,0) [below] {1/2};
		\node at (4,0) [below] {1};
		\node at ({4/sqrt(3)},0) {$\bullet$}
		node at ({4/sqrt(3)},0) [below right] {$1/\sqrt{3}$};
		\draw [*->,shorten <=-2.4pt] (0,.2) -- (4,.2);
		\draw [*-{*[fill=white]},shorten >=-2.4pt,shorten <=-2.4pt] (0,.4) -- (2,.4) ;
		\draw [*-{*[fill=white]},shorten >=-2.4pt,shorten <=-2.4pt] (0,.6) --  (2,.6) ;
		\draw [*-{*[fill=white]},shorten >=-2.4pt,shorten <=-2.4pt] (2,.8) -- node [above] {} ({4/sqrt(3)},.8)   ;
		\draw (0,.1) -- (-.1,.1) -- node[left]{$H_0$} (-.1,.7) -- (0,.7);
		\draw (0,.75) -- (-.1,.75) -- node[left]{$H_1$} (-.1,.9) -- (0,.9);
		\end{tikzpicture}
	\end{figure}
\end{example}

\noindent
For more details about persistent homology, see~\cite{CDGO16, carlsson2014}.

\subsection{Magnitude homology and persistent homology}\label{SS:pers magn}

Magnitude homology is a homology theory for metric spaces that was introduced as a categorification of magnitude for finite graphs in \cite{HW17} and for finite metric spaces in \cite{LS17}.
Magnitude homology and persistent homology were first related in \cite{O18}. In that paper, the author introduced `blurred magnitude homology', which is the homology of a filtered simplicial set called the `enriched nerve'. 

\begin{definition}
Let $(X,d)$ be a metric space. The \define{blurred magnitude homology} of $(X,d)$ is the homology of the filtered simplicial set with set of $n$-simplices given by 
\[
N(X)(\epsilon)_n=\left \{(x_0,\dots , x_n)\mid \sum_{i=0}^{n-1}d(x_i,x_{i+1})\leq \epsilon \right\}
\]
for each $\epsilon\in \mathbb{R}_{\geq 0}$. The filtered simplicial set $N(X)$ is the \define{enriched nerve} of $X$.
\end{definition}

In \cite[Theorem 4.9]{GH21} the authors show that for any finite metric space $(X,d)$, the magnitude $|tX|$ can be recovered from the barcode of the blurred magnitude homology of $X$. More precisely, they show that for $t$ large enough one has
\begin{equation*}\label{Eq:magn blurred ph}
|tX|=\sum_{k=0}^\infty\sum_{i=1}^{m_k} (-1)^k(e^{-a_{k,i}t}-e^{-b_{k,i}t}),
\end{equation*}
where $\{[a_{k,i},b_{k,i})\}_{i=1}^{m_k}$ is the barcode of the homology in degree $k$ of $N(X)$. 
Motivated by this connection, the authors then proceed to define a new invariant for persistence modules.

\begin{definition}
Let $M$ be a finitely presented persistence module with barcode decomposition in  degree $k$ given by $\{[a_{k,i},b_{k,i})\}_{i=1}^{m_k}$. The \define{persistent magnitude} of $M$ is the real number
\begin{equation*}\label{E:pm}
|M|=\sum_{k=0}^\infty \sum_{i=1}^{m_k}(-1)^k(e^{-a_{k,i}}-e^{-b_{k,i}})\, .
\end{equation*}
The \define{persistent magnitude function} of $M$ is the function
\begin{alignat}{2}
(0,\infty)&\notag\to \mathbb{R}\\
t&\notag\mapsto |tM|\, .
\end{alignat}
\end{definition}

Given a metric space $(X,d)$, Govc and Hepworth then proceed to define Rips magnitude as the magnitude of the persistence module obtained by computing the persistent homology of the Vietoris--Rips complex\cite{{GH21}}.

\begin{definition}
Let $(X,d)$ be a finite metric space. Its \define{Rips magnitude} is the persistent magnitude of the persistent homology of $X$ with respect to the Vietoris--Rips complex:
\[
|X|_{\mathrm{Rips}}=|H_\star(\VR(X))|\, .
\]
The \define{Rips magnitude function} is the function

\begin{alignat}{2}
(0,\infty)&\notag\to \mathbb{R}\\
t&\notag\mapsto |tX|_{\mathrm{Rips}}\, .
\end{alignat}
\end{definition}

\subsection{Fractal dimensions}

 Fractals are used to model complex physical phenomena such as eroded coastlines or snowflakes, in which similar patterns recur at progressively smaller scales, and in describing phenomena such as fluid turbulence, crystal growth, and galaxy formation. One measure of the complexity of a fractal is its fractal dimension, which can be thought of as an extension of the standard notion of dimension beyond integers. A point ought to have dimension $0$, a line dimension $1$, a solid in $\RR^2$ dimension $2$, and so on. This is perhaps best understood by the observation that expanding one of these token spaces by a scalar $r$ results in an expansion of the volume by a factor of $r^d$ where $d$ is the dimension of the space that we consider. 
 Yet we frequently encounter spaces such as the Cantor set or the Koch snowflake which, in some sense, seem to exist between integer dimensions. That is, an expansion of the middle-thirds Cantor set by $3$ results in a space which contains in essence two copies of the Cantor set we have begun with, hence twice the volume.  Measures of fractal dimension seek to quantify and expand upon this intuition. Here we examine two fractal dimensions in connection to magnitude: the Minkowski dimension and the Hausdorff dimension.
 \begin{figure}[h!]
 \centering
 \begin{tikzpicture}[decoration=Cantor set, line width=10, scale=2.3]
  \draw decorate{ (0,0) -- (3,0) };
  \draw decorate{ decorate{ (0,-.5) -- (3,-.5) }};
  \draw decorate{ decorate{ decorate{ (0,-1) -- (3,-1) }}};
  \draw decorate{ decorate{ decorate{ decorate{ (0,-1.5) -- (3,-1.5) }}}};
\end{tikzpicture}
\caption{Level sets of the middle thirds Cantor set: the $0$th level (not depicted here) is given by the unit interval; the $1$st level is given by removing the middle third from the unit interval, and is depicted on top of the figure, the $2$nd level is given by removing the middle thirds from the intervals in the 1st level, and is depicted right below, and so on. Refer to the statement of Proposition \ref{cantoralpha} for a rigorous definition. Note the self similarity between the level sets; each contains two copies of the prior level set scaled by $1/3$.}
\end{figure}
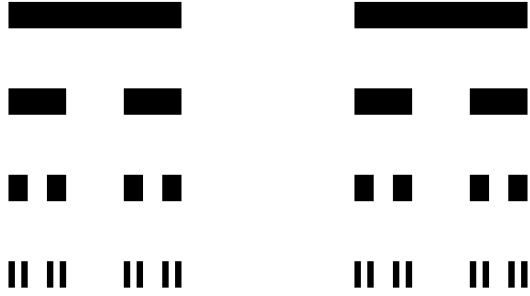
 
\begin{definition}
	Let $\epsilon>0$. Denote by $N(X,\epsilon)$ the minimum number of closed balls of radius $\epsilon$ needed to cover $X$. Then the  \textbf{upper Minkowski dimension} or box-counting dimension of $X$ is defined to be $$\overline{\text{dim}}_{Mink}X\colon=\limsup\limits_{\epsilon\to0}\dfrac{\log N(X,\epsilon)}{\log 1/\epsilon}, $$ and the \textbf{lower Minkowski dimension} $$\underline{\text{dim}}_{Mink}X\colon=\liminf\limits_{\epsilon\to 0}\dfrac{\log N(X,\epsilon)}{\log 1/\epsilon}. $$ We say that the \textbf{Minkowski dimension} of $X$ exists if the following limit $$\text{dim}_{Mink}X\colon=\lim\limits_{\epsilon\to0}\dfrac{\log N(X,\epsilon)}{\log 1/\epsilon} $$ exists. In this case, the upper, lower, and Minkowski dimensions of $X$ are equivalent.
	\end{definition}

\begin{definition} Let $X\subset \RR^n$. 
Define $H_r^d(X)$ to be the number
\[\inf \left\{\sum\limits_{i\in I} r_i^d\colon \text{ X is covered by countably many balls of diameter }0<r_i<r  \textrm{ for $i\in I$}  \right\},    \] and let \[ H^d(X)=\lim\limits_{r\to 0} H_r^d(X). \]
The {\bf Hausdorff dimension} of the set $X$ is 
\[\textrm{dim}_H(X)=\inf\{ d\geq 0 \,|\,H^d(X)=0\}.
\]
\end{definition}

In general, $\dim_{Mink}(X)\geq\dim_H(X)$. This follows since if we suppose 
\[
d>\liminf \frac{\log N(X,\epsilon)}{\log 1/\epsilon}=\underline{\text{dim}}_{Mink}X\, ,
\]
then for a cover of $X$ by $N(X,\epsilon)$ balls $B_i$ of radius $\epsilon/2$, we have \[ S_\epsilon=\sum\limits_{i=1}^{N(X,\epsilon)}\textrm{diam}(B_i)^d=N(X,\epsilon)\epsilon^d=\epsilon^{d-\frac{\log N(X,\epsilon)}{\log 1/\epsilon}}   \] so that $\inf\limits_{\epsilon>0} S_\epsilon=0$. Thus $d$ cannot be the Hausdorff dimension if $d$ is larger than the lower Minkowski dimension \cite{CB17}.

\begin{remark}
While Minkowski and Hausdorff dimension are two prominent sorts, other dimensions such as correlation and information dimension exist and are often used to estimate fractal dimension. We refer the reader to \cite{PG83} for further information on these concepts.  
\end{remark}

\begin{example}
 We take a look at Minkowski and Hausdorff dimensions for some standard subsets of $\RR$ and $\RR^2$.
    \begin{itemize}
        \item Let $I$ be the unit interval. Then $\dim_{Mink}(I)=\dim_H(I)=1$. Intuitively, one would expect the limit $\lim\limits_{\epsilon\to0}\dfrac{\log N(I,\epsilon)}{\log 1/\epsilon}$ to be $1$, as the relationship between $1/\epsilon$ and $N(I,\epsilon)$ is linear. That is, it takes $1/\epsilon$ many $\epsilon$ balls to cover the unit interval.
        \item Let $\SSS^1$ denote the unit circle with metric inherited from $\RR^2$. Then 
        \[
        \dim_{Mink}(\SSS^1)=\dim_H(\SSS^1)=1.
        \] The intuition for this example is identical to the unit interval. 
        \item Let $C$ denote the middle-thirds Cantor set. Then $\dim_{Mink}(C)=\dim_H(C)=\frac{\log2}{\log 3}$. This is most easily seen through self-similarity. That is, if we make the space three times as large, we need two times as many balls to cover it. Hence, we would expect $\lim\limits_{\epsilon\to0}\dfrac{\log N(C,\epsilon)}{\log 1/\epsilon}=\log_32=\frac{\log2}{\log3}$.
        \item Let $K$ denote the Koch snowflake, see \cite{CB17}. Then $\dim_{Mink}(K)=\dim_H(K)=\frac{\log 4}{\log 3}$. The rationale is similar to that of the Cantor set. That is, if the Koch snowflake is made three times as large, the perimeter can be seen to be $4$ times as long.
        \item Let $\QQ\cap I=Q$. Then $\dim_{Mink}(Q)=1$ but $\dim_{H}(Q)=0$. Here the Hausdorff dimension is easily seen to be $0$, as for any value of $d>0$ we simply give each point its own ball. One can see  that $\dim_{Mink}(Q)=\dim_{Mink}(I)$ by the consideration that we cannot remove any of the balls counted in $N(I,\epsilon)$ as doing so would uncover some point of $Q$. 
    \end{itemize}
   Detailed proofs of these statements may be found in \cite{CB17}. 
\end{example}

\section{Alpha magnitude}
 As discussed in Section \ref{SS:pers magn}, in \cite{GH21} the authors define Rips magnitude as the magnitude of the persistent homology of the Vietoris--Rips complex of a finite metric space. 
 However, this is limited by the complexity of computing the persistent homology of the Vietoris--Rips complex up to a degree given by the cardinality of the space, and we therefore 
 introduce alpha and \v{C}ech magnitude, as the magnitude of the persistent homology of the alpha and \v{C}ech complexes associated to a finite set of points in Euclidean space. As illustrated by Proposition \ref{P:alpha magn}, this notion of persistent magnitude is better suited for computational purposes, because we only need to compute persistent homology up to degree $d$, where $d$ is the  dimension of the ambient space.  We then proceed to study some properties of these invariants and define them for a larger class of compact subsets of Euclidean space. 

\subsection{Alpha and \v{C}ech magnitude of finite subspaces of Euclidean space}

\begin{definition}
Let $(X, d)$ be a finite metric subspace of Euclidean space $\mathbb{R}^d$.
 The \define{alpha magnitude} of $(X,d)$ is the persistent magnitude of the persistent homology of $X$ with respect to the alpha complex:
\[
|X|_{\Alpha}=|H_\star(\Alpha(X))|\, .
\]
The \define{alpha magnitude function} is the function

\begin{alignat}{2}
(0,\infty)&\notag\to \mathbb{R}\\
t&\notag\mapsto |tX|_{\Alpha}\, .
\end{alignat}
\end{definition}

\begin{definition}
The \define{\v{C}ech magnitude} of $(X,d)$ is the persistent magnitude of the persistent homology of $X$ with respect to the \v{C}ech complex:
\[
|X|_{\textrm{\v{C}ech}}=|H_\star(\textrm{\v{C}ech}(X))|\, .
\]
The \define{\v{C}ech magnitude function} is the function

\begin{alignat}{2}
(0,\infty)&\notag\to \mathbb{R}\\
t&\notag\mapsto |tX|_{\textrm{\v{C}ech}}\, .
\end{alignat}
\end{definition}

\begin{lemma}\label{L:alpha and cech}
Let $(X, d)$ be a finite metric subspace of Euclidean space $\mathbb{R}^d$. Then for any $t\in \mathbb{R}_{\geq 0}$ we have $|tX|_{\Alpha}=|tX|_{\textrm{\v{C}ech}}$.
\end{lemma}

\begin{proof}
By the nerve lemma we know that both $\Alpha(X)_\epsilon$ as well as $\textrm{\v{C}ech}(X)_\epsilon$ have the same homotopy-type as the union of balls with radius $\epsilon$ centered at the points in $X$. Thus, we have that
$\Alpha(X)_\epsilon\simeq \textrm{\v{C}ech}(X)_\epsilon$ for any $\epsilon\in \mathbb{R}_{\geq 0}$. 
\end{proof}

We proceed to study properties of the alpha magnitude;
by Lemma \ref{L:alpha and cech} such properties also hold for the \v{C}ech magnitude. First, we restate a proposition from \cite{GH21}.

\begin{proposition}\label{P:alpha magn}
Let $(X, d)$ be a finite metric subspace of $ \mathbb{R}^d$ and assume that the $k$th barcode for $(X, d)$ constructed from alpha complexes is $\{[a_{k,i},b_{k,i})\}_{i=1}^{m_k}$. Then the alpha magnitude function of $(X, d)$ is given by:

    \begin{equation} |tX|_{\Alpha}=\sum_{k=0}^d \sum_{i=0}^{m_k}(-1)^k(e^{-a_{k,i}t}-e^{-b_{k,i}t})\, .
    \end{equation}
 Furthermore, the alpha magnitude captures the `effective number of points' of a space, i.e, we have that  
 \begin{equation}\lim_{t\to 0}|tX|_{\Alpha}=1
 \end{equation}
 and 
 \begin{equation}
     \lim_{t\to \infty}|tX|_{\Alpha}=\# X\, .
  \end{equation}
\end{proposition}

\begin{proof}
For (1), by the nerve lemma $\Alpha(X)_\epsilon\simeq \cup_{x\in X}B_\epsilon(x)$, and therefore $H_k(\Alpha(X)_\epsilon)\cong 0$ for all $k>d$ and all $\epsilon\in \mathbb{R}_{\geq 0}$.
For (2) and (3), this follows from the properties of the alpha complex, and Proposition 5.10 in \cite{GH21}.
\end{proof}

\begin{example}\label{E:alpha mag finite p}
	We consider the alpha magnitude of a finite set of points $X=\{x_0,x_1,\dots,x_n\}$ in $\RR$. We assume that $x_0\leq x_1 \leq \ldots \leq x_n$. Then the alpha complex on $X$ will have non-trivial homology only in dimension $0$. To see what the expression for alpha magnitude would be in in this case, consider two adjacent points $x_i$, $x_{i+1}\subset X$. Then when $\epsilon\geq d(x_i,x_{i+1})/2$ we have that $\sigma=\{x_i,x_{i+1}\}\in \alpha(X)_\epsilon$. One can easily see that this holds for the remaining points, and thus, 
	$$ |X|_\alpha=n+1-\sum\limits_{k=0}^{n-1}e^{-\frac{d(x_k,x_{k+1})}{2}}.$$
	\end{example}

We  note the following relationship between alpha magnitude and Rips magnitude of finite subsets of the real line.
	
\begin{lemma}\label{Prop:RipsAlphaFinRelation}
    Let $A\subset \RR$, $A$ finite. Then
    \[ |A|_\alpha=\left|\begin{frac}{1}{2}\end{frac}A\right|_{\mathrm{Rips}} \,.\]
\end{lemma}
\begin{proof}
   Let $A=\{x_0,\dots , x_n\}\subset \mathbb{R}$.
    According to Proposition 7.6 in \cite{GH21}, 
    \[ |tA|_{\mathrm{Rips}}=n+1-\sum_{j=0}^{n-1} e^{-d(x_{j+1},x_j)t}\,. 
    \] 
     Our expression in  Example \ref{E:alpha mag finite p} for $|A|_\alpha$ implies the result.

\end{proof}	

\subsection{Alpha magnitude of compact subspaces of Euclidean space}

We have so far only considered alpha magnitude for finite sets of points in $\RR^n$. We now propose the following as an extension to compact subsets of $\RR^n$.

\begin{definition}
Let $X\subset\RR^n$ be a compact metric subspace. Then if the following limit exists for all sequences of finite subsets of $A$ converging to $X$ with respect to the Hausdorff metric, we define it to be the \define{alpha magnitude} of the space $X$:

$$|X|_\alpha=\lim\limits_{\#(A)<\infty,\,A\subset X}|A|_\alpha.$$
\end{definition}

\begin{remark} 
In \cite{GH21} Rips magnitude for compact spaces is prospectively defined in two ways. First, in terms of a `lower' and an `upper' definition replacing the limit above with a limit infimum and a limit supremum respectively, as well as a definition akin to the one we propose above, where the lower and upper definitions agree. We choose to employ the most restrictive version here. 
Further, as we show below in Proposition \ref{circlealpha}, in all the examples that we consider where alpha magnitude exists, the upper and lower version of alpha magnitude agree. We leave for future work to study the relationship between the upper and lower definitions. 

\end{remark}

We next observe a relationship between the alpha magnitude and Rips magnitude of an interval of the real line.

\begin{lemma}\label{unitint}
    Let $B=[a, b]\subset\RR$ be an interval in $\RR$. Then 
    \[
    |B|_{\alpha}=\left|\frac{1}{2}B \right|_{\mathrm{Rips}}\, , 
    \]  
    and in particular
    \[ 
    |B|_\alpha=1+\frac{b-a}{2}\, . 
    \]
\end{lemma}

\begin{proof}
We note that the alpha magnitude  of $B$ is defined to be the limit of the alpha magnitudes over all finite subsets of $B$. Any finite subset $A\subset [a,b]$ has alpha magnitude equivalent to the Rips magnitude of $\frac{1}{2}A\subseteq \left[\frac{1}{2}a,\frac{1}{2}b\right]$ by Lemma \ref{Prop:RipsAlphaFinRelation}. Hence, we may write 
\[
\left|B \right|_\alpha=\lim_{A\subset B,\#A<\infty}|A|_\alpha=\lim_{A\subset B,\#A<\infty}\left|\frac{1}{2}A\right|_{\mathrm{Rips}}=\left|\frac{1}{2}B \right|_{\mathrm{Rips}} \, .
\]
Theorem 9.1 in \cite{GH21} states that the Rips magnitude of an interval $tB=[ta,tb]$ is 
\[ 
|tB|_{\mathrm{Rips}}=1+t(b-a)\,. 
\]
The second statement follows by setting $t=(1/2)$ 
\[ |B|_\alpha=\left|\frac{1}{2}B \right|_{\mathrm{Rips}}=1+\frac{b-a}{2}\,.  
\]

\end{proof}

We now observe a monotonicity property for the alpha magnitude of subsets of $\RR$.

\begin{proposition}\label{monotonicity}
    Let $A\subseteq B\subset\RR$, $A,B$ finite. Then $|A|_\alpha\leq|B|_\alpha$. Further, if $A\subset L$, $L$ an interval of finite length $l$, then $|A|_\alpha\leq 1+\frac{l}{2}$.
\end{proposition}

\begin{proof}
    Lemma \ref{unitint} gives an expression for alpha magnitude in terms of Rips magnitude for intervals. Theorem 9.1 in \cite{GH21} states that for finite $A\subseteq B\subset\RR$, 
    \[ 
    |A|_{\mathrm{Rips}}\leq|B|_{\mathrm{Rips}}\,. 
    \]
    Then 
    \[
    |A|_\alpha=\left|\frac{1}{2}A \right|_{\mathrm{Rips}}\leq \left|\frac{1}{2}B \right|_{\mathrm{Rips}}=|B|_\alpha\, .
    \] 
    The second conclusion follows since 
    \[
    |A|_{\alpha}=\lim_{A'\subset A,\#A'<\infty}|A'|_{\alpha}\leq\lim_{A'\subset L,\#A'<\infty}|A'|_{\alpha}=1+\frac{l}{2} \, .
    \]
\end{proof}

We now prove the following result for the alpha magnitude of a finite union of intervals. This argument follows closely from the proof of \cite[Theorem~9.1]{GH21}.

\begin{proposition}
\label{alphaints}
Let $A\subset [0,1/2]$ be a finite union of $n\in\NN$ disjoint intervals, $A=\bigcup_{i=1}^nA_i$, where $A_i=(a_{i,1},a_{i,2})$ and $a_{1,1}< a_{1,2}< a_{2,1} < a_{2,2} < \ldots < a_{n,1}< a_{n,2}$. Let $g_i=a_{i+1,1}-a_{i,2}$ for $1\leq i\leq n-1$ and let $l_i=a_{i,2}-a_{i, 1}$ for $1\leq i\leq n$. Then 
$$|tA|_\alpha=1+\sum_{i=1}^n tl_i/2+\sum_{j=1}^{n-1}(1-e^{-tg_j/2}).$$
\end{proposition}

\begin{proof}
Since $A \subset [0,1/2]$, $\sum_{i=1}^{n-1}g_i\leq 1/2$ and $\sum_{i=1}^nl_i\leq 1/2$. Suppose $B$ consisting of $${b_{1,1} \leq \ldots \leq b_{1,m_1} \leq b_{2,1}\leq \ldots \leq b_{n,m_n}}$$ is a finite set of points such that $d_H(B,A)<\delta<\min\limits_{1\leq i\leq n-1}\dfrac{g_i}{2}$. 

Define $f\colon (0,\infty)\to\RR$ by\small \begin{align*} f(t)&=1+\sum_{i=1}^n tl_i/2+\sum_{j=1}^{n-1}(1-e^{-tg_j/2})-|tB|_\alpha\\
&=1+\sum_{i=1}^n t l_i/2+\sum_{j=1}^{n-1}(1-e^{-tg_j/2})-(\sum_{i=1}^n m_i-\sum_{i=1}^n\sum_{j=1}^{m_i-1}e^{(b_{i,j}-b_{i,j+1})t/2}-\sum_{j=1}^{n-1}e^{(b_{j,m_j}-b_{j+1,1})t/2})\\
&=1+\sum_{i=1}^n t l_i/2+n-1-\sum_{j=1}^{n-1}e^{-tg_j/2}-(\sum_{i=1}^n m_i-\sum_{i=1}^n\sum_{j=1}^{m_i-1}e^{(b_{i,j}-b_{i,j+1})t/2}-\sum_{j=1}^{n-1}e^{(b_{j,m_j}-b_{j+1,1})t/2})\\
&=\sum_{i=1}^n t l_i/2+n-\sum_{i=1}^n m_i-\sum_{j=1}^{n-1}e^{-tg_j/2}+\sum_{i=1}^n\sum_{j=1}^{m_i-1}e^{(b_{i,j}-b_{i,j+1})t/2}+\sum_{j=1}^{n-1}e^{(b_{j,m_j}-b_{j+1,1})t/2}\\
&=\sum_{i=1}^n t l_i/2+n-\sum_{i=1}^n m_i+\sum_{i=1}^n\sum_{j=1}^{m_i-1}e^{(b_{i,j}-b_{i,j+1})t/2}+\sum_{j=1}^{n-1}(e^{(b_{j,m_j}-b_{j+1,1})t/2}-e^{-tg_j/2}).\\
\end{align*}
\normalsize
We define $|0B|_\alpha=1$ to extend $f$ to a function on $[0,\infty)$. We first note $f(0)=0$. Then since $d_H(B,A)<\delta$,
\begin{align*}f'(0)&=\sum_{i=1}^n l_i/2-\sum_{i=1}^n(b_{i,m_i}-b_{i,1})/2+\sum_{j=1}^{n-1}(b_{j,m_j}-b_{j+1,1}+g_j)/2\\
&=\sum_{i=1}^nl_i/2-\sum_{i=1}^n(b_{i,m_i}-b_{i,1})/2)+\sum_{i=1}^{n-1}(b_{i,m_i}-b_{i+1,1})/2+\sum_{i=1}^{n-1}g_i/2\\
&=(\sum_{i=1}^nl_i/2+\sum_{i=1}^{n-1}g_i/2)+\sum_{i=1}^{n-1}(b_{i,m_i}-b_{i+1,1})/2-\sum_{i=1}^n(b_{i,m_i}-b_{i,1})/2)\\
&=(a_{n,2}-a_{1,1})/2+(b_{1,1}-b_{n,m_n})/2\\
&=(b_{1,1}-a_{1,1}+a_{n,2}-b_{n,m_n})/2\\&\leq \delta\leq 2n\delta\end{align*}
Then,  \begin{align*}
    f''(0)&=\sum_{i=1}^n\sum_{j=1}^{m_i-1}(b_{i,j}-b_{i,j+1})^2/4+\sum_{j=1}^{n-1}((b_{j,m_j}-b_{j+1,1})^2/4-(g_j)^2/4)\\&\leq \sum_{i=1}^n\sum_{j=1}^{m_i-1}\delta(b_{i,j+1}-b_{i,j})/2+\sum_{j=1}^{n-1}(-b_{j,m_j}+b_{j+1,1}-g_j)(-b_{j,m_j}+b_{j+1,1}+g_j)/4\\&\leq\delta/2+\sum_{j=1}^{n-1}\delta(-b_{j,m_j}+b_{j+1,1}+g_j)/2\\&\leq \delta/2+\delta(2(n-1)\delta+1)/2=\delta+(n-1)\delta^2\leq n\delta.
\end{align*}
The first inequality follows since for each $1\leq i\leq n$ and $1\leq j\leq m_i-1$, we have $b_{i,j+1}-b_{i,j}<2\delta$ since $d_H(B,A)<\delta$. We note that $\sum_{i=1}^n\sum_{j=1}^{m_i-1}(b_{i,j+1}-b_{i,j})\leq\sum_{i=1}^nl_i\leq 1$ and $-b_{j,m_j}+b_{j+1,1}-g_j\leq 2\delta $ hold. Since $d_H(B,A)<\delta$, we conclude the second inequality holds. The third inequality follows since $\delta< \min\limits_{1\leq i\leq n-1}\dfrac{g_i}{2}$ and thus,\[ \sum_{j=1}^{n-1}(-b_{j,m_j}+b_{j+1,1}+g_j)\leq \sum_{j=1}^{n-1}(g_j+2\delta+g_j)\leq 1+2(n-1)\delta.  \] For the last inequality, we note that $\delta<1$ implies $\delta^2<\delta$. It follows that $$0\leq f(t)\leq n\delta(2t+t^2/2).$$ Hence, as $\delta\to 0$, $f\to 0$ uniformly on any bounded subset of $[0,\infty)$, so that $$\lim_{B\to A,\, B \textrm{ finite}}|tB|_\alpha=1+\sum_{i=1}^n tl_i/2+\sum_{j=1}^{n-1}(1-e^{-g_j/2})$$ and so $$|tA|_\alpha=1+\sum_{i=1}^n tl_i/2+\sum_{j=1}^{n-1}e^{-g_j/2}$$ by our definition, as desired.
\end{proof}

\begin{proposition}\label{cantoralpha}
	We consider the middle-thirds Cantor set, $C$, defined as follows:
	\[C=\bigcap\limits_{i=0}^\infty C_i \text{, where } C_i=\bigcup\limits_{k=0}^{2^i-1}\left[\frac{2k}{3^i},\frac{2k+1}{3^i}\right]. \]

We have that the alpha magnitude of the Cantor set exists, and further \[
|C|_\alpha=1+\sum\limits_{j=0}^{\infty}2^{j}\left(1-e^{-(1/2)(1/3)^{j+1}}\right)\, .
\]

\begin{proof}
	We denote by $K_n$ the $n$-th level set of $C$, $\bigcap\limits_{i=0}^n C_i$.
	First, we note that the alpha complex of $C$ has non-trivial homology at most in dimension 0. Suppose $\{H_i\}_{i\in\NN}$ is a sequence of finite subsets of $C$ converging to $C$ in the Hausdorff metric. Then select $H_i$ such that $d_H(H_i,C)<\delta\in (0, 1/3)$. Then for our choice of $\delta$, there is some $n\in\NN$ such that $(1/3)^{n+1}\leq2\delta\leq(1/3)^{n}$. Our method of proof is as follows: \begin{enumerate}
	    \item We establish an upper bound on $|H_i|_\alpha$ with a covering of $C$, hence a covering of $H_i$, so that \ref{monotonicity} implies the bound.
	    \item We establish a lower bound on $|H_i|_\alpha$ through manipulation of the expression for $|H_i|_\alpha$.
	    \item We demonstrate these bounds have the same limit as $\delta\to 0$, and use the Sandwich Theorem \cite[Section 8.1 Sequences, p.\ 557]{JS09} to argue that the alpha magnitude of $\{H_i \}_{i\in\NN}$ has a limit.

	\end{enumerate} We first define the thickening of each set $C_i$ as follows:
	\[C^*_i=\bigcup\limits_{k=0}^{2^i-1}\left[\frac{2k}{3^i}-\delta,\frac{2k+1}{3^i}+\delta\right]. \]
	Next, we observe that the thickenings stabilise for $i>n$. That is, since $2\delta>(1/3)^{n+1}\geq (1/3)^i$, it follows that $\dfrac{2k+1}{3^i}+\delta>\dfrac{2(k+1)}{3^i}-\delta$, the endpoint of the next interval. Hence, for $i>n$, \[C^*_i=[-\delta,1+\delta]=C^*_0. \]
	Let
	\[
	C^*=\bigcap\limits_{i=0}^\infty C^*_i =\bigcap\limits_{i=0}^n C^*_i.
	\]
	Hence we may consider $C^*$ to be a $\delta$ thickening of the $n$-th level set $K_n$ of $C$. 
	Then since $K_n$ has $2^n$ intervals of length $(1/3)^n$, $C^*$ has $2^n$ intervals of length $(1/3)^n+2\delta$. Hence, half the sum of the lengths of intervals in $C^*$ is $\frac{2^{n-1}}{3^n}+2^n\delta$. We now observe that $K_n$ has precisely $2^{j-1}$ gaps of length $(1/3)^j$ for each $1\leq j\leq n$. Hence, $C^*$ has $2^{j-1}$ gaps of length $(1/3)^j-2\delta$ for each $1\leq j\leq n$. By Proposition~\ref{monotonicity}, $$|H_i|_\alpha\leq|C^*|_\alpha=1+\frac{2^{n-1}}{3^n}+2^n\delta+\sum\limits_{j=1}^{n}2^{j-1}(1-e^{-(1/2)(1/3)^j+\delta}), $$ the equality following from \ref{alphaints} and our above characterisation. This completes (1). \\ 
	
	We now proceed to prove (2). First, we note each interval in $C^*$ contains at least one point of $H_i$ since $d_H(H_1,C)<\delta\in\RR$ and $C^*$ is a delta covering of $C$. Since we have a precise characterisation of the gaps in $C^*$, order $C^*=\bigcup_{j=1}^{2^{n}}I_j$, where each $I_j$ is an interval, and where if $c\in I_j$, $c'\in I_k$, and $j<k$, then $c<c'$. Then assign points in $H_i$ the following indices: \[H_i=\{h_{1,1},\dots,h_{1,m_1},h_{2,1},\dots,h_{2^{n},m_{2^{n}}} \}   \] where each $h_{j,k}\in C_j$. Then we write the following expression for $|H_i|_\alpha$: \[|H_i|_\alpha=1+\sum\limits_{j=1}^{2^{n}}\sum\limits_{k=1}^{m_{j}-1}(1-e^{(h_{j,k}-h_{j,k+1})/2})+\sum\limits_{j=1}^{2^{n}-1}(1-e^{(h_{j,m_j}-h_{j+1,1})/2})   \]
	Now, by our above characterisation of $C^*$, each term $h_{j,m_j}-h_{j+1,1}\leq-(1/3)^k+2\delta<0$ for some $k$. Hence, \begin{align*}|H_i|_\alpha&\geq 1+\sum\limits_{j=1}^{2^{n}}\sum\limits_{k=1}^{m_{j}-1}(1-e^{(h_{j,k}-h_{j,k+1})/2})+\sum\limits_{k=1}^{n}2^{k-1}(1-e^{-(1/2)(1/3)^{k}+\delta})\\&> 1+\sum\limits_{k=1}^{n}2^{k-1}(1-e^{-(1/2)(1/3)^{k}+\delta})=L\end{align*} where the second inequality follows through removing a positive term, and we label the final expression $L$ for brevity. This completes (2).\\
	
To prove (3) note that $$|C^*|_\alpha-L=\frac{2^{n-1}}{3^n}+2^n\delta<\frac{2^{n-1}}{3^n}+2^{n-1}\frac{1}{3^{n}}=\frac{2^n}{3^n}.$$ As $\delta\to 0$, $n\to \infty$, hence $|C^*|_\alpha-L\to 0$. But since $|C^*|_\alpha$ and $L$ are upper and lower bounds on $|H_i|_\alpha$, respectively, $|H_i|_\alpha$ must approach the same limit by the Sandwich theorem.
	We now show the above expression converges. We have shown $$\lim\limits_{\delta\to 0}|H_i|_\alpha=\lim\limits_{\delta\to 0}|C^*|_\alpha=1+\sum\limits_{j=0}^{\infty}2^{j}(1-e^{-(1/2)(1/3)^{j+1}}).$$ This converges, since  $$\lim\limits_{k\to\infty}\dfrac{2^{k+1}(1-e^{-(1/2)(1/3)^{k+2}})}{2^k(1-e^{-(1/2)(1/3)^{k+1}})}=\lim\limits_{k\to\infty}2\dfrac{e^{(1/2)(1/3)^{k+2}}-1}{e^{(1/2)(1/3)^{k+2}}-e^{-(1/3)^{k+2}}}=2/3$$ after an application of l'Hopital's rule.  This completes (3).

	\end{proof}

		\end{proposition}
	
	We now turn to the example of the circle considered in Euclidean space with the standard Euclidean metric. Govc and Hepworth \cite{GH21} demonstrate that the Rips magnitude cannot be said to exist in the sense we define above as the limit of Euclidean cycles inscribed in the circle does not exist. Here we demonstrate that the limit of such cycles does in fact exist for alpha magnitude.

\begin{lemma}\label{chordcont}
Let $\{P_n\}_{n\in\NN}$ denote a sequence of finite subsets of $\SSS^1$ converging to $\SSS^1$ in the Hausdorff metric. Let $P_i=\{p_1,\dots,p_{n_i}\}$ be some element of this sequence with the points listed in clockwise order. Define 
\[
M(P_i)=\max\{d(p_0, p_1), d(p_1, p_2), \ldots, d(p_{n_i-1},p_0 )\},\] where $d$ is the distance function inherited from $\RR^2$. Then as $\{P_n\}_{n\in\NN}$ converges to $\SSS^1$ in the Hausdorff metric, $M(P_n)$ converges to $0$. 
    
\end{lemma}

\begin{figure}[h!]
\begin{center}
    \begin{tikzpicture}[scale=0.7][>=latex]
  \begin{axis}[xmin=-1.1, xmax=0, ymin=0, ymax=1.1,
                axis x line = bottom,
                axis y line = right,
               axis equal, xlabel style={right}, ylabel style={above},
               xtick={-1,-.5,0}, ytick={0,.5,1}, xlabel={$x$}, ylabel={$y$},
               yticklabels={0,.5,1},
               xticklabels={-1,-.5,0}]
    \draw[black,thick] (110,0) circle (52mm);
    \addplot[thick, dashed, color=red, samples=100, domain=-.9505:-.1837] {.87676*(x+.9505)+.3107};
    \addplot[thick, dashed, color=red, samples=100, domain=-.9505:-.6592] {1.5149*(x+.9505)+.3107};
    \addplot[thick, dashed, color=red, samples=100, domain=-.6592:-.5671] {-1.13937*(x+.6592)+.7519};
    \addplot[only marks, mark=*, red] coordinates {(-.1837,.9830)};
    \addplot[only marks, mark=*, red] coordinates {(-.9505,.3107)};
    \node[red] at (5, 30) {$p_i$};
    \node[red] at (85, 105) {$p_{i+1}$};
    \node[red] at (65, 60) {$M(P_i)$};
    \node[red] at (52, 72) {$s$};
    \node[red] at (37,59) {$\delta$};
    \end{axis}
  \end{tikzpicture}
  \caption{$M(P_i)$, the sagitta $s$, and $\delta$.}
\end{center}
\end{figure}
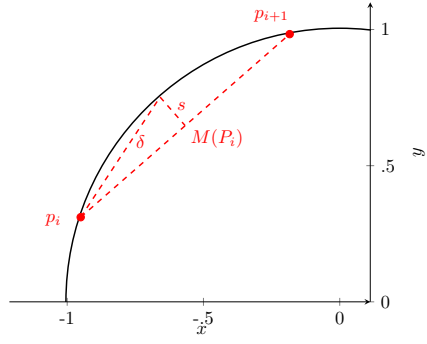

\begin{proof}
Let $\epsilon>0$. We will show that there exists some $N$ such that for $i>N$, $M(P_i)<\epsilon$. Since $\{P_n\}_{n\in\NN}$ converges to $\SSS^1$ in the Hausdorff metric, there exists some $N$ such that for $i>N$, $d(P_i,\SSS^1)<\delta$. Hence, let $0<\delta<\frac{\epsilon}{2}$. $M(P_i)$ represents the greatest chord length between points of $P_i$ on $\SSS^1$, so it is realized as a chord between two consecutive points of $P_i$ in the clockwise order. Then $d_H(P_i,\SSS^1)$ will be represented by the hypotenuse of the right triangle constructed from $M(P_i)/2$ and the sagitta $s$.

We use the formula for chord length in terms of the sagitta and with radius $1$ to write the expression $$M(P_i)=2\sqrt{2s-s^2} $$ and combine this with the basic pythagorean observation $$\frac{M(P_i)^2}{4}+s^2=\delta^2.$$ From these two equations we can express  $M(P_i)$ as a function of $\delta$
\begin{equation}\label{chordal distance}
    M(P_i)=\delta\sqrt{4-\delta^2}
\end{equation} Then, \begin{align*}
    M(P_i)&\leq 2\delta\leq \epsilon
\end{align*} where the first inequality follows since $\delta>0$ and the second by our assumption $\delta<\epsilon/2$. Thus, we have shown $M(P_n)\to 0$ as $d_H(P_n,\SSS^1)\to 0$. \vspace{1.5mm}

\end{proof}	
	
	We make use of the following statement from \cite[Theorem 22.6]{lay2007convex}.
	
	\begin{lemma}\label{percont}
	    For any Hausdorff converging sequence $K_n\to K_\infty$ of convex, compact sets
	    \[
	    \textrm{perim} K_\infty=\lim_{n\to\infty} \textrm{perim} K_n.
	    \]
	    With other words, the perimeter function is continuous with respect to the Hausdorff distance on the space of compact, convex sets.
	\end{lemma}

\begin{figure}[h!]
\begin{center}
    \begin{tikzpicture}[>=latex][scale=0.7]
  \begin{axis}[xmin=-1.1, xmax=1.1, ymin=-1.1, ymax=1.1,
               grid=both, grid style=dotted, axis lines=middle,
               axis equal, xlabel style={right}, ylabel style={above},
               xtick={-1,-.5,0,.5,.1}, ytick={-1,-.5,0,.5,1}, xlabel={$x$}, ylabel={$y$},
               yticklabels={-1,-.5,0,.5,.1},
               xticklabels={-1,-.5,0,.5,.1}]
    \draw[black,thick] (110,110) circle (26mm);
    \addplot[thick, dashed, color=red, samples=100, domain=-.9505:-.7390] {-4.6544*(x+.9505)+.3107};
    \addplot[only marks, mark=*, red] coordinates {(1,0)};
    \addplot[only marks, mark=*, red] coordinates {(.5,0.866)};
    \addplot[only marks, mark=*, red] coordinates {(-.558,0.8299)};
    \addplot[only marks, mark=*, red] coordinates {(.7159,0.6982)};
    \addplot[only marks, mark=*, red] coordinates {(-.7390,-0.6737)};
    \addplot[only marks, mark=*, red] coordinates {(0.9840,-.1781)};
    \addplot[only marks, mark=*, red] coordinates {(-.2903,-.9569)};
    \addplot[only marks, mark=*, red] coordinates {(0.8362,0.5484)};
    \addplot[only marks, mark=*, red] coordinates {(-.1837,.9830)};
    \addplot[only marks, mark=*, red] coordinates {(-.9505,.3107)};
    \addplot[only marks, mark=*, red] coordinates {(0.4720,-.8816)};
    \addplot[only marks, mark=*, red] coordinates {(-.5051,.8631)};
    \node[red] at (40, 100) {$d_n$};
    \end{axis}
  \end{tikzpicture}
  \caption{Points from $P_i$.}
\end{center}
\end{figure}
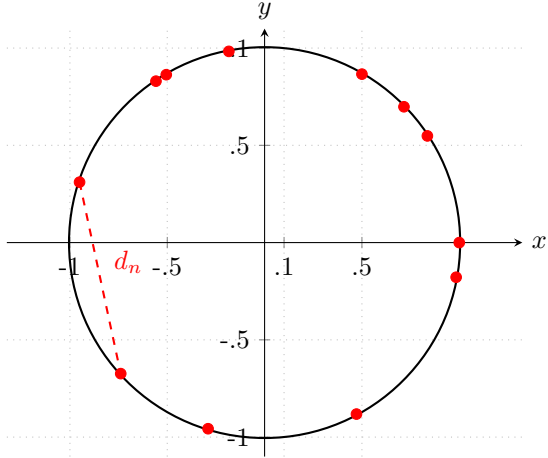

\begin{proposition}\label{circlealpha}
Let $(\SSS^1,d)$ be the metric space with points from $\SSS^1$ and the distance function inherited from $\RR^2$. Then \[|t\SSS^1|_\alpha=\pi t+e^{-t}\]. 
\end{proposition}

\begin{proof}
We proceed as follows: \begin{itemize}
    \item We establish the alpha magnitude function of an arbitrary finite subset $P_i$ of $\SSS^1$ which has $d_H(P_i,\SSS^1)<\delta$.
    \item We show that, if finite subsets $\{P_n\}_{n\in\NN}$ converge to $\SSS^1$, then the difference between $|P_i|_\alpha$ and $\pi t+e^{-t}$ is uniformly close to $0$.
\end{itemize} Let $\{P_n\}_{n\in\NN}$ be an arbitrary sequence of finite subsets of $\SSS^1$ converging to $\SSS^1$ in the Hausdorff metric. Let $P_i=\{p_1,p_2,\dots,p_n \}$ be such an element of this sequence that $d_H(P_i,\SSS^1)<\delta<1$. Denote by $d_j=d(p_{j},p_{j+1})$ for $1\leq j\leq n-1$ and $d_n=d(p_n,p_1)$. Without loss of generality assume that $d_n=\max\limits_{1\leq j\leq n}d_j$ (this is easily achieved by a relabeling of $P_i$). The alpha complex for $P_i$ has nontrivial homology only for $H_0$ and $H_1$. In $H_0$, there are $n-1$ bars of length $d_j/2$ for $1\leq j\leq n-1$. There is no bar of length $d_n/2$, as by this time there is only one connected component. There is also one bar which begins at time $0$ and does not die. In $H_1$, there is just one bar beginning at time $d_n/2$, which dies at time $1$. Hence, by the definition of alpha magnitude, we have \[|P_i|_\alpha=1+\left(\sum\limits_{i=1}^{n-1}(1-e^{-d_it/2})\right)-(e^{-d_nt/2}-e^{-t}).\] This completes (1).

For (2), we write the following expression: $$f(t)=\pi t+e^{-t}-1-\left(\sum\limits_{i=1}^{n-1}(1-e^{-d_it/2})\right)+(e^{-d_nt/2}-e^{-t}). $$
 We argue this expression is uniformly close to $0$ as $\delta\to 0$. Since $\{P_n\}_{n\in\NN}$ converges to $\SSS^1$ in the Hausdorff metric, it also converges in the perimeter by \ref{percont}. Further, each chordal distance $d_j\to 0$ as $\delta\to 0$ by Lemma~\ref{chordcont}. Hence, we may select $\delta_0,\delta_1$ such that \[\max\limits_{1\leq j\leq n}d_j<\delta_0,\quad\pi-\sum\limits_{i=1}^n\frac{d_i}{2}<\delta_1.\] As $\delta\to 0$, $\max\{\delta_0,\delta_1\}\to 0$. Then, we simplify $f(t)$ to $$f(t)=\pi t-\left(\sum\limits_{i=1}^{n}(1-e^{-d_it/2})\right).$$
Then $f(0)=0$. Further, 
$$f'(0)=\pi-\sum\limits_{i=1}^n\frac{d_i}{2},$$
and so 
$$0\leq f'(0)\leq \delta_1,$$
where the second inequality follows from our upper bound by $\delta_1$. The second derivative and bounds on it are given by $$0\leq f"(0)=\sum\limits_{i=1}^n(\frac{d_i^2}{4})\leq\delta_0/2\sum\limits_{i=1}^n\frac{d_i}{2}\cdot 1\leq \begin{frac}{\delta_0}{2}\end{frac}\pi.$$
where the second inequality follows from our bound on chordal distance, and the third inequality follows from our bound on perimeter. Thus, 
\[
0\leq f(t)\leq \pi(\max\{\delta_0,\delta_1\})(\frac{t}{2}+\frac{t^2}{4}),\] 
and so $f(t)\to 0$ uniformly as $\delta\to 0$ on any bounded subset of $\RR$. This completes (2). Hence, as before, the alpha magnitude of $\SSS^1$ is $\pi t+e^{-t}$.

\end{proof}
	
	We now consider an example where the alpha magnitude can be shown to not exist.

\begin{figure}
\begin{center}
    \begin{tikzpicture}[>=latex][scale=0.6]
  \begin{axis}[xmin=-0.1, xmax=1.1, ymin=-0.1, ymax=1.1,
               grid=both, grid style=dotted, axis lines=middle,
               axis equal, xlabel style={right}, ylabel style={above},
               xtick={0,.2,.4,.6,.8,1}, ytick={0,.2,.4,.6,.8,1}, xlabel={$x$}, ylabel={$y$},
               yticklabels={0,.2,.4,.6,.8,1},
               xticklabels={0,.2,.4,.6,.8,1}]
    \addplot[only marks, mark=*, red] coordinates {(0,0)};
    \addplot[only marks, mark=*, red] coordinates {(0,0.2)};
    \addplot[only marks, mark=*, red] coordinates {(0,0.4)};
    \addplot[only marks, mark=*, red] coordinates {(0,0.6)};
    \addplot[only marks, mark=*, red] coordinates {(0,0.8)};
    \addplot[only marks, mark=*, red] coordinates {(0,1)};
    \addplot[only marks, mark=*, red] coordinates {(0.2,0)};
    \addplot[only marks, mark=*, red] coordinates {(0.2,0.2)};
    \addplot[only marks, mark=*, red] coordinates {(0.2,0.4)};
    \addplot[only marks, mark=*, red] coordinates {(0.2,0.6)};
    \addplot[only marks, mark=*, red] coordinates {(0.2,0.8)};
    \addplot[only marks, mark=*, red] coordinates {(0.2,1)};
    \addplot[only marks, mark=*, red] coordinates {(0.4,0)};
    \addplot[only marks, mark=*, red] coordinates {(0.4,0.2)};
    \addplot[only marks, mark=*, red] coordinates {(0.4,0.4)};
    \addplot[only marks, mark=*, red] coordinates {(0.4,0.6)};
    \addplot[only marks, mark=*, red] coordinates {(0.4,0.8)};
    \addplot[only marks, mark=*, red] coordinates {(0.4,1)};
    \addplot[only marks, mark=*, red] coordinates {(0.6,0)};
    \addplot[only marks, mark=*, red] coordinates {(0.6,0.2)};
    \addplot[only marks, mark=*, red] coordinates {(0.6,0.4)};
    \addplot[only marks, mark=*, red] coordinates {(0.6,0.6)};
    \addplot[only marks, mark=*, red] coordinates {(0.6,0.8)};
    \addplot[only marks, mark=*, red] coordinates {(0.6,1)};
    \addplot[only marks, mark=*, red] coordinates {(0.8,0)};
    \addplot[only marks, mark=*, red] coordinates {(0.8,0.2)};
    \addplot[only marks, mark=*, red] coordinates {(0.8,0.4)};
    \addplot[only marks, mark=*, red] coordinates {(0.8,0.6)};
    \addplot[only marks, mark=*, red] coordinates {(0.8,0.8)};
    \addplot[only marks, mark=*, red] coordinates {(0.8,1)};
    \addplot[only marks, mark=*, red] coordinates {(1,0)};
    \addplot[only marks, mark=*, red] coordinates {(1,0.2)};
    \addplot[only marks, mark=*, red] coordinates {(1,0.4)};
    \addplot[only marks, mark=*, red] coordinates {(1,0.6)};
    \addplot[only marks, mark=*, red] coordinates {(1,0.8)};
    \addplot[only marks, mark=*, red] coordinates {(1,1)};
    \draw [decorate,decoration={brace,amplitude=8pt},xshift=0pt,yshift=63pt]
(0,0.4) -- (0,25) node [black,midway,xshift=-5mm] 
{\large$\frac{1}{5}$};

    \end{axis}
  \end{tikzpicture}
  \caption{The graph of $A_5$.}
\end{center}
\end{figure}
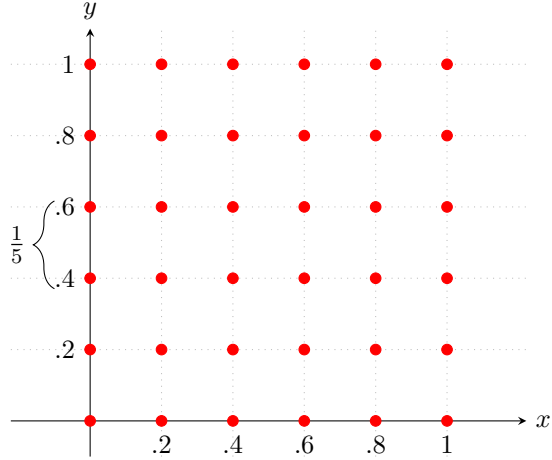
\begin{example}\label{ex:unitsquare}
Let $I^2\subset\RR^2$ be the closed unit square. 
Denote by $\{A_n\}_{n\in\NN}$ a sequence of finite subsets of $I^2$ such that 
\[
A_n=\bigcup\limits_{j=0}^n\bigcup_{k=0}^n\left\{\left(\frac{j}{n},\frac{k}{n}\right)\right\}.    
\]
Then the sequence $\{A_n\}_{n\in\NN}$ converges to $I^2$ in the Hausdorff metric as $i\to\infty$. The alpha complex for $A_n$ has non-trivial homology only for $H_0$ and $H_1$ as a subset of $\RR^2$. For $H_0$, there are $(n+1)^2-1$ bars starting at time $0$ and which die at time $\frac{1}{2n}$, and one bar which starts at time $0$ and does not die. For $H_1$, there are $n^2$ bars starting at time $\frac{1}{2n}$ and dying at time $\frac{1}{\sqrt{2}n}$. Hence, using our definition of alpha magnitude, we compute the alpha magnitude of $A_n$ as follows:

\[
|A_n|_\alpha=1+((n+1)^2-1)(1-e^{-\frac{1}{2n}})-n^2(e^{-\frac{1}{2n}}-e^{-\frac{1}{\sqrt{2}n}}),
\]
and simplify to 
\begin{align*}
|A_n|_\alpha&=1+(n^2+2n)(1-e^{-\frac{1}{2n}})-n^2(e^{-\frac{1}{2n}}-e^{-\frac{1}{\sqrt{2}n}}) \\
&=1+n^2(1-2e^{-\frac{1}{2n}}+e^{-\frac{1}{\sqrt{2}n}})+2n.
\end{align*}
Then, with the immediate observations $2e^{-\frac{1}{2n}}<2$ and $e^{-\frac{1}{\sqrt{2}n}}>1-\frac{1}{\sqrt{2}n}$, we have \begin{align*}
    |A_n|_\alpha\geq 1+n^2\left(1-2+1-\frac{1}{\sqrt{2}n}\right)+2n=1+n\left(2-\frac{1}{\sqrt{2}}\right)
\end{align*}
 which diverges as $n\to\infty$. Hence, it cannot be the case that alpha magnitude exists for the unit square $I^2$, as the limit specified in the definition does not exist for every such convergent sequence.

\end{example}

\section{Alpha magnitude dimension}\label{S:alpha magnitude dimenion}

Motivated by Meckes' work\cite{Meckes2015} in which he related the magnitude dimension with the Minkowski dimension, here we study the relationship between alpha magnitude and fractal dimensions. We first recall Meckes' work, and we then define the alpha magnitude dimension, a number associated to the alpha magnitude of a finite metric space. 

\subsection{Magnitude dimension}\label{SS:mag dim}

Here we recall the results by 
Meckes \cite{Meckes2015}. He first defines the magnitude dimension.

\begin{definition}[Meckes]
	The \textbf{upper magnitude dimension} of $X$ is defined to be $$\overline{\dim}_{\Mag}X\colon=\limsup\limits_{t\to\infty}\dfrac{\log|tA|}{\log t} \,. $$
	The \textbf{lower magnitude dimension} of $X$ is defined to be $$\underline{\dim}_{\Mag}X\colon=\liminf\limits_{t\to\infty}\dfrac{\log|tA|}{\log t}\, . $$ 
	When $$\dim_{\Mag}X\colon=\lim\limits_{t\to\infty}\dfrac{\log|tA|}{\log t} $$ exists, it is equivalent to the previous two expressions, and we define this to be the \textbf{magnitude dimension} of $X$. 
	\end{definition}

Meckes proceeds to prove the following result, relating the magnitude dimension with the Minkowski dimension of compact subsets of Euclidean space.
\begin{theorem}(\cite[Corollary 7.4]{Meckes2015})\label{T:Meckes}
	If $X\subset \RR^n$ is compact and if $\dim_\Mag X$ or $\dim_\Mink X$ exists, then both exist and $\dim_\Mag X=\dim_\Mink X$.
	\end{theorem}

Using the result from this theorem, one could use the magnitude to approximate the Minkowski dimension, as done for instance  by Willerton\cite{W09}. However, we note that this method has substantial limitations: the computational complexity of magnitude makes computations prohibitive for many of the spaces that we are interested in in practice, and also limits substantially the size of subsamples that we can take. We note that the complexity of the computation of magnitude depends on the complexity of matrix inversion, which is known to be bounded below by $\Omega(n^2\log n)$ for an invertible  $n\times n$ matrix \cite{RR02}.

\subsection{Alpha magnitude dimension}\label{SS:alpha mag dim}
Motivated by the relationship between the magnitude dimension and the Minkowski dimension, we use alpha magnitude to define a similar notion of dimension.

\begin{definition}
Let $(X, d)$ be  a metric space with an everywhere-defined alpha magnitude function. We define the \define{alpha magnitude dimension} of $(X, d)$ to be the non-negative real number
\[\text{dim}_{\alpha}X\colon=\lim\limits_{t\to\infty}\dfrac{\log|tX|_{\Alpha}}{\log t} \, ,
\]
whenever this limit  exists. 
\end{definition}

\begin{remark}
We note that the assumption that the alpha magnitude function is defined everywhere is not a limitation for the examples that we consider in our work. The definition and results  in this work may be relaxed to account for spaces with alpha magnitude functions that are not defined for at most a finite number of points.
\end{remark}

\begin{example}
For any finite set $X\subset\RR^d$, we have $|X|_\alpha\leq \#X$. Hence,  $$\lim\limits_{t\to\infty}\dfrac{\log |tX|_\alpha}{\log t}=0 $$ and so $\dim_{\alpha}(X)=0$.
\end{example}

\begin{example}
Let $I\subset\RR$ denote the unit interval. By Example \ref{unitint} we know that $|tI|_\alpha=1+\frac{t}{2}$. 
Then we have 
\[
\lim\limits_{t\to\infty}\dfrac{\log (\frac{t}{2}+1)}{\log t}=1\, ,
\]
and hence $\dim_{\alpha}(I)=1.$
\end{example}

\begin{example}\label{E:amd cantor}
	We  computed  the alpha magnitude of the middle thirds Cantor set $C$ in Proposition  \ref{cantoralpha}. We know that 
	\[
	|C|_\alpha=1+\sum\limits_{j=0}^{\infty}2^{j}\left(1-e^{-(1/2)(1/3)^{j+1}}\right)\, .\]
	A similar argument shows 
	\[
	|3C|_\alpha=1+\sum\limits_{j=0}^{\infty}2^{j}\left(1-e^{-(1/2)(1/3)^{j}}\right)\]
	and so

	\[|3C|_\alpha=2|C|_\alpha-1+(1-e^{-1/2})=2|C|_\alpha-e^{-1/2}\, . \]
	Repeating this process for $|3^nC|_\alpha$ we have that \[
	2^n|C|_\alpha-\sum\limits_{k=0}^{n-1}2^{n-k-1}e^{-(1/2)(3^k)}=|3^nC|_\alpha.\]
	Then 
	\[
	\lim\limits_{n\to\infty}\dfrac{\log\left(2^n|C|_\alpha-\sum\limits_{k=0}^{n-1}2^{n-k-1}e^{-(1/2)(3^k)}\right)}{\log(3^n)} \]
	is an expression approaching the alpha magnitude dimension of the Cantor set. We now introduce the following bounds: 
	$$\dfrac{\log(2^n|C|_\alpha-2^{n}e^{-1/2})}{\log(3^n)} \leq \dfrac{\log\left(2^n|C|_\alpha-\sum\limits_{k=0}^{n-1}2^{n-k-1}e^{-(3^k)}\right)}{\log(3^n)} \leq \dfrac{\log(2^n|C|_\alpha-2^{n-1}e^{-1/2})}{\log(3^n)}. $$
	The larger expression removes terms from the negative sum. The smaller expression sends all the terms in the negative sum to $e^{-1/2}$ which is at least as large as the rest. Further, we observe that 
	\[
	|C|_\alpha>1>e^{-1/2}\, .
	\]
	Since the limit of both the expression on the left and the right after factoring is ${\log(2)}/{\log(3)}$, we conclude that the alpha magnitude dimension of the Cantor set is equivalent to its Minkowski dimension by the Sandwich Theorem. 
\end{example}

\begin{example}\label{E: amd circle}
    We compute the alpha magnitude of the circle in Proposition \ref{circlealpha}, where we show that $|t\SSS^1|_\alpha=\pi t+e^{-t}$. Then  $$ \lim\limits_{t\to\infty}\dfrac{\log|t\SSS^1|}{\log t}=1,$$ which matches the Minkowski dimension of the circle. 
\end{example}

The previous examples lead us to state the following conjecture, which is further supported by the heuristic approximations that we perform in Section \ref{S:experiments}.
\begin{conjecture}
    For $X\subset \RR^n$ compact, when $\text{dim}_{\alpha}X$ exists, then \[\dim_{\alpha}X=\dim_{\Mink}(X) . \]
\end{conjecture}

\section{Experiments}\label{S:experiments}

Here we show how the alpha magnitude can be used to estimate the Minkowski dimension from samples. We first estimate the dimensions of spaces for which we have already verified that the alpha magnitude dimension exists. We then proceed to estimate the alpha magnitude dimension of the Feigenbaum attractor, for which it is not known how to compute fractal dimensions precisely, but several different estimates have been put forward.

We note that there have been several efforts in using  persistent homology to develop methods to estimate known fractal dimensions, we refer the reader to the comparative study \cite{JS20} of several of these efforts and references therein for further details. The reader may  refer to \cite{PG83,gneiting} and references therein for additional estimates of dimension that do not make use of persistent homology.

\subsection{Methodology}\label{SS:methodology}

Given a metric space $X$, our aim is to estimate the  limit

\begin{equation}\label{E:limit alpha mag}
\lim_{t\to\infty}\frac{\log |tX|_\alpha}{\log t}\, . \end{equation}

\noindent
 We proceed by studying the following expression for a sequence of finite metric spaces $\{X_n\}$ sampled uniformly at random approaching $X$ in the Hausdorff metric

\begin{equation}
    \label{E:limit approx alpha mag}
\lim_{t\to\infty}\lim_{n\to\infty}\frac{\log |tX_n|_\alpha}{\log t}. \end{equation}

\begin{figure}[ht!]
\begin{alignat}{2}
&\notag
    \includegraphics[scale=.35]{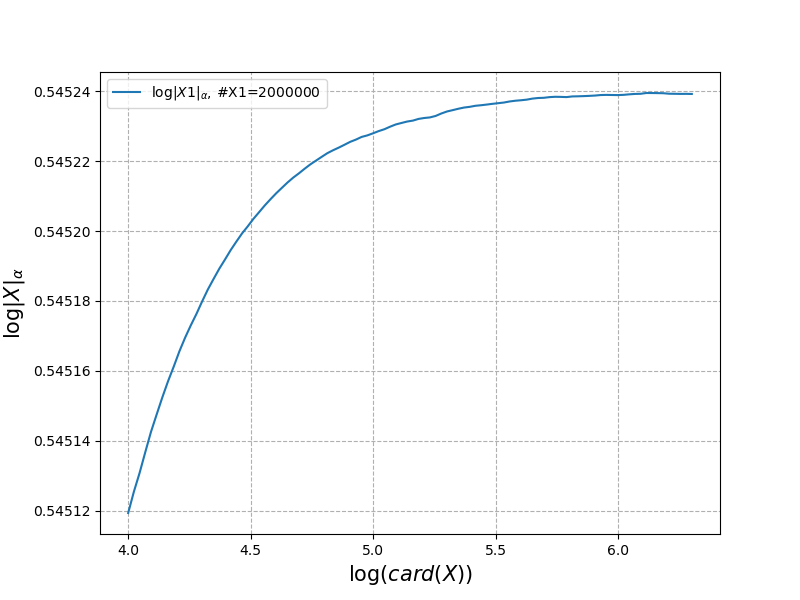}
&&\notag  \includegraphics[scale=.35]{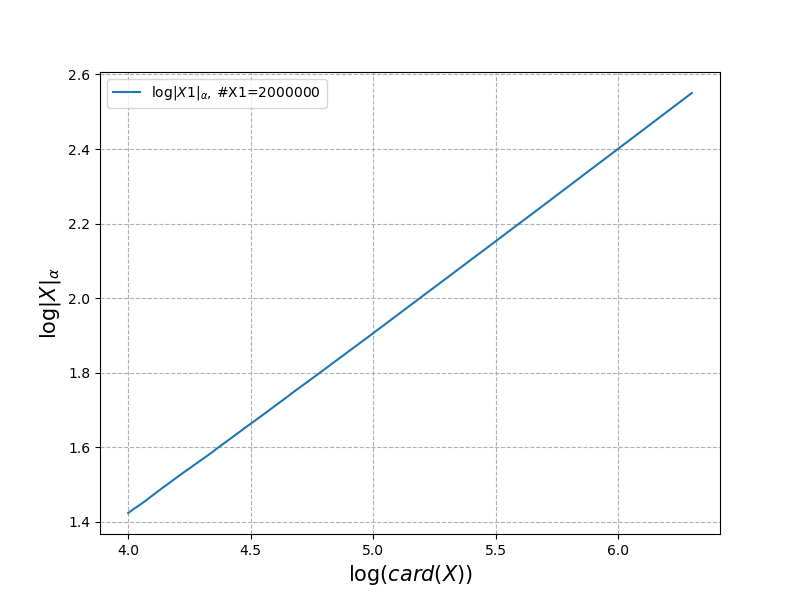}\\
    &\notag\qquad \qquad \qquad \qquad\qquad \qquad(a)
    &&\notag \qquad \qquad\qquad \qquad\qquad \qquad(b)\\
    &\notag
    \includegraphics[scale=.35]{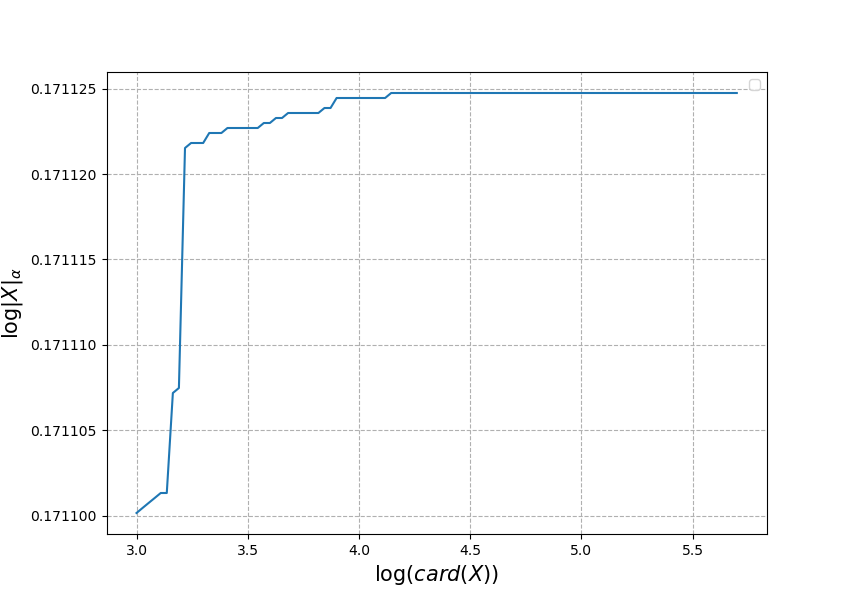}
&&\notag \includegraphics[scale=.35]{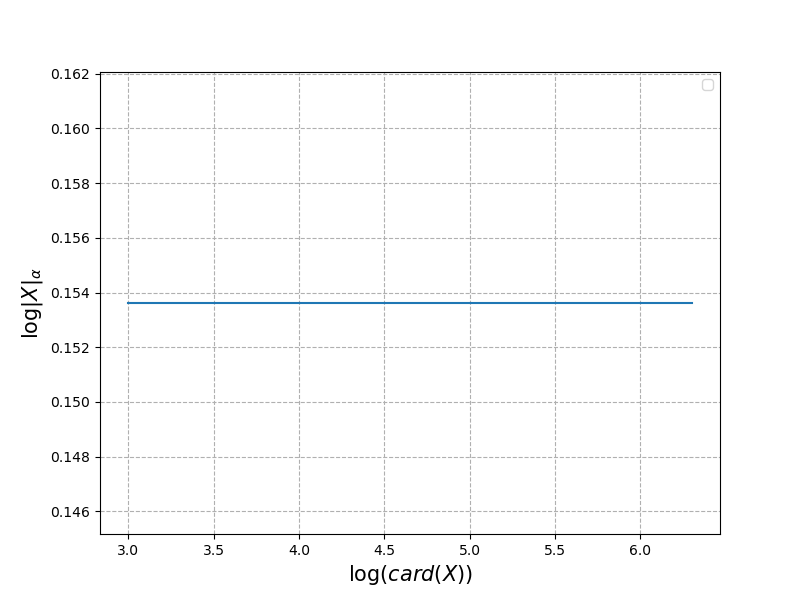}\\
    &\notag\qquad \qquad \qquad \qquad\qquad \qquad(c)
    &&\notag \qquad \qquad\qquad \qquad\qquad \qquad(d)
\end{alignat}
\caption{
(a) 
      We reproduce a plot of the  alpha magnitude of a given sample of points  from the unit circle against the cardinality of the sample under a  logarithmic scale. 
     (b)  We reproduce a plot of the  alpha magnitude of a given sample of points  from the unit square against the cardinality of the sample under a  logarithmic scale. The growth of the alpha magnitude compared to the sample size taken from the unit square can be seen to not only fail to `settle down', but even seems to grow linearly under the log scale. (c)-(d) are plots for the Cantor set and Feigenbaum attractor respectively, in the fashion of (a).
}
\label{fig:circex}
  \end{figure}

\noindent
For this, we take $n$ sufficiently large so that $\left||tX_n|_\alpha - |tX_{n-1}|_\alpha\right |$ is small enough so that we may consider $|tX_n|_\alpha$ representative of $|tX|_\alpha$. For the datasets that we consider, this means that we choose $n$ such that  $\left||tX_n|_\alpha - |tX_{n-1}|_\alpha\right |<1\cdot10^{-5}$ for $t=1$.  
 In  Figure \ref{fig:circex} a) we provide an illustration of this convergence for the circle. We note that the alpha magnitude of $\SSS^1\subset\RR^d$ exists, as we have proven occurs for sufficiently dense sequences of points with respect to Hausdorff distance taken from the unit circle, see Proposition \ref{circlealpha}. We illustrate a lack of convergence in Figure \ref{fig:circex} (b) for samples taken from the unit square, that is, a space for which the alpha magnitude does not stabilise, see Example \ref{ex:unitsquare}. 
  
 Once we have chosen an $n$ so that $|tX_n|_\alpha$ may be considered a representative of $|tX|_\alpha$, we proceed to  plot the numerator $\log |tX_n|_\alpha$ against the denominator $\log t$ and apply a linear regression plot to determine whether the expression follows a power rule for the interval we determine to be representative. This can be seen as, in some sense, an application of l'Hopital's rule, as we are essentially evaluating the limit by estimating the slope. For the circle and Cantor set, we determine the interval $[1.75,\log(n)-2]$ to best represent a following of a power rule through observation. However, the Feigenbaum attractor requires different bounds to best observe a power rule for the purpose of estimating dimension, and so we conclude that an appropriate choice of bounds depends largely on the underlying data set. 

\begin{figure}[ht!]
\begin{alignat}{2}
&\notag
    \includegraphics[scale=.4]{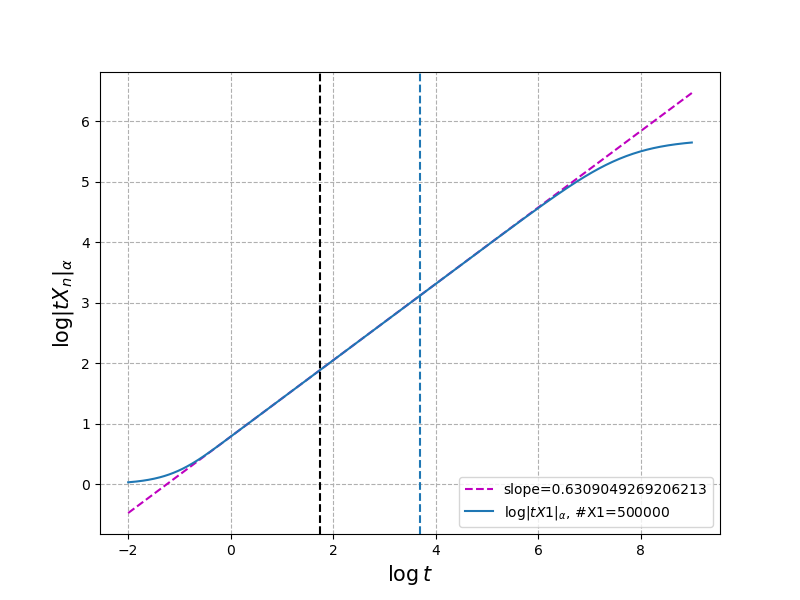}
&&\notag  \includegraphics[scale=.4]{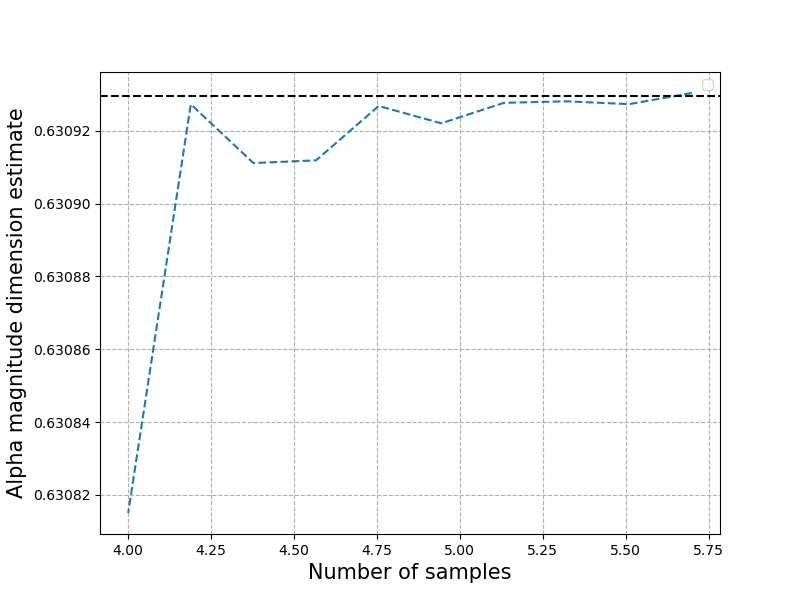}\\
    &\notag\qquad \qquad \qquad \qquad\qquad \qquad(a)
    &&\notag \qquad \qquad\qquad \qquad\qquad \qquad(b)\\
    &\notag
    \includegraphics[scale=.4]{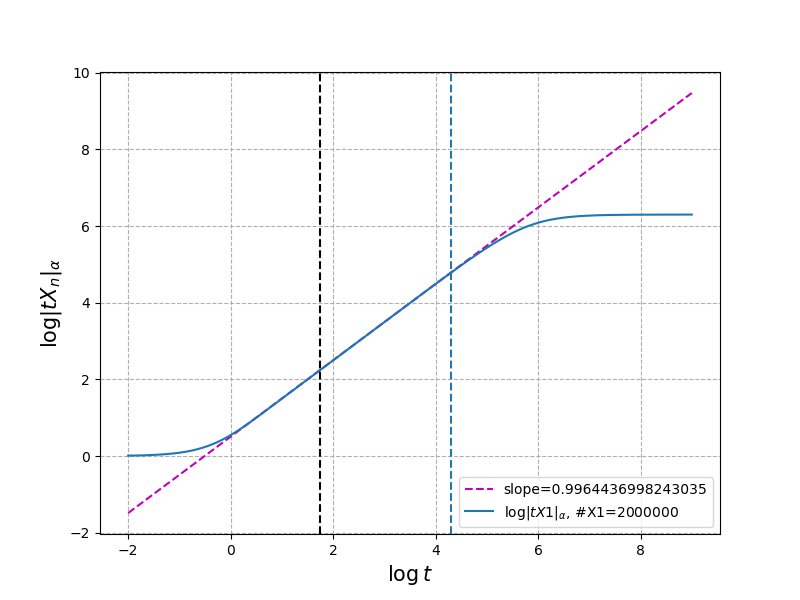}
&&\notag \includegraphics[scale=.4]{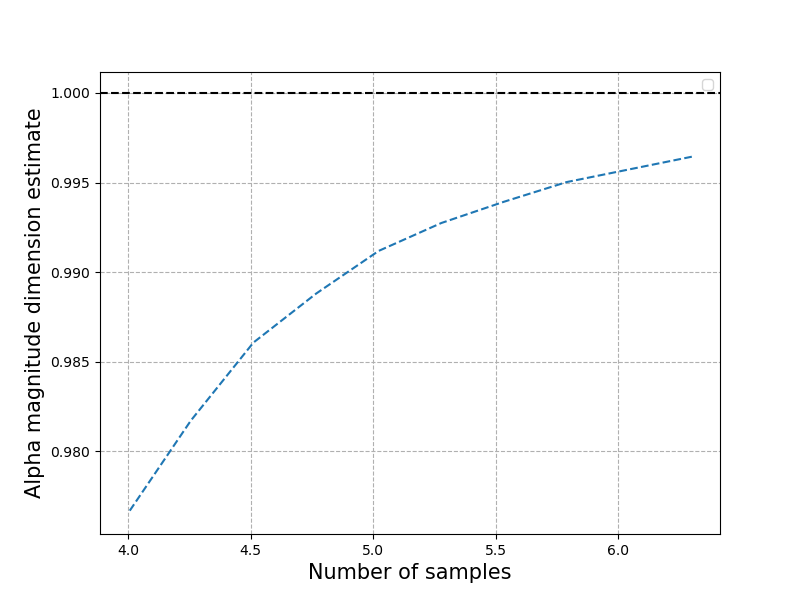}\\
    &\notag\qquad \qquad \qquad \qquad\qquad \qquad(c)
    &&\notag \qquad \qquad\qquad \qquad\qquad \qquad(d)
\end{alignat}
\caption{
(a) Log-log alpha magnitude plot for a sample $X$ of $5*10^5$ points sampled from the Cantor set at depth level $100$. The vertical dashed lines represent the left and right endpoint of the range over which we estimate the exponent of the power law through linear regression. 
(b) Plot of the estimate resulting from sample sizes spaced logarithmically. 
(c) Log-log alpha magnitude plot for a sample of $2*10^6$ points from $\SSS^1$. The estimate is close to the true alpha magnitude dimension $1$. (d) Different estimates of the alpha magnitude dimension of $\SSS^1$ resulting from subsamples of different sizes.
}
\label{cantor-circle}
  \end{figure}

\subsection{Cantor Set}

We compute the alpha magnitude of the middle-thirds Cantor set in Proposition 
 \ref{cantoralpha}, and its alpha magnitude dimension in Example \ref{E:amd cantor}. We demonstrate the methodology  from Section \ref{SS:methodology} by obtaining an estimate of the alpha magnitude dimension of the Cantor set computed from a sample of points. We sample $5*10^5$ points uniformly at random  from the 100th level of the middle-thirds Cantor set. For each point, sampling is done through coin flips to determine which interval in each level of depth the point resides in, then a number in the closed interval $[0,1]$ is selected uniformly at random to determine where in the interval the point is. 
 
 In Figure \ref{cantor-circle}(a) we illustrate the log-log alpha magnitude plot for our sample of points $X$, while in Figure \ref{cantor-circle}(b) we illustrate different estimates of the alpha magnitude dimension for subsamples of $X$ of different sizes. We note that as the size of the subsamples increases, the estimates become more accurate. We provide details about how these estimates were obtained though the methodology discussed at the beginning of this section  in Figure \ref{fig:cant-circ-multiregplot}. Here we can observe the progression of the alpha magnitudes of each subsample, as well as ten corresponding logarithmically spaced ranges over which we perform the power-law estimation. 
The slope that we obtain from our estimations is very accurate: it is approximately $0.6309$ for all sizes of subsamples greater than $10^4$, while it is  approximately $0.6308$ for the sample of size $10^4$.  The actual alpha magnitude dimension of the Cantor set is $\frac{\log 2}{\log 3} \approx 0.6309$. Larger data sets roughly correspond with better estimates of dimension, though all estimates are highly accurate.

\begin{figure}[h!]
 \begin{alignat}{2}
  \notag &\includegraphics[scale=.31]{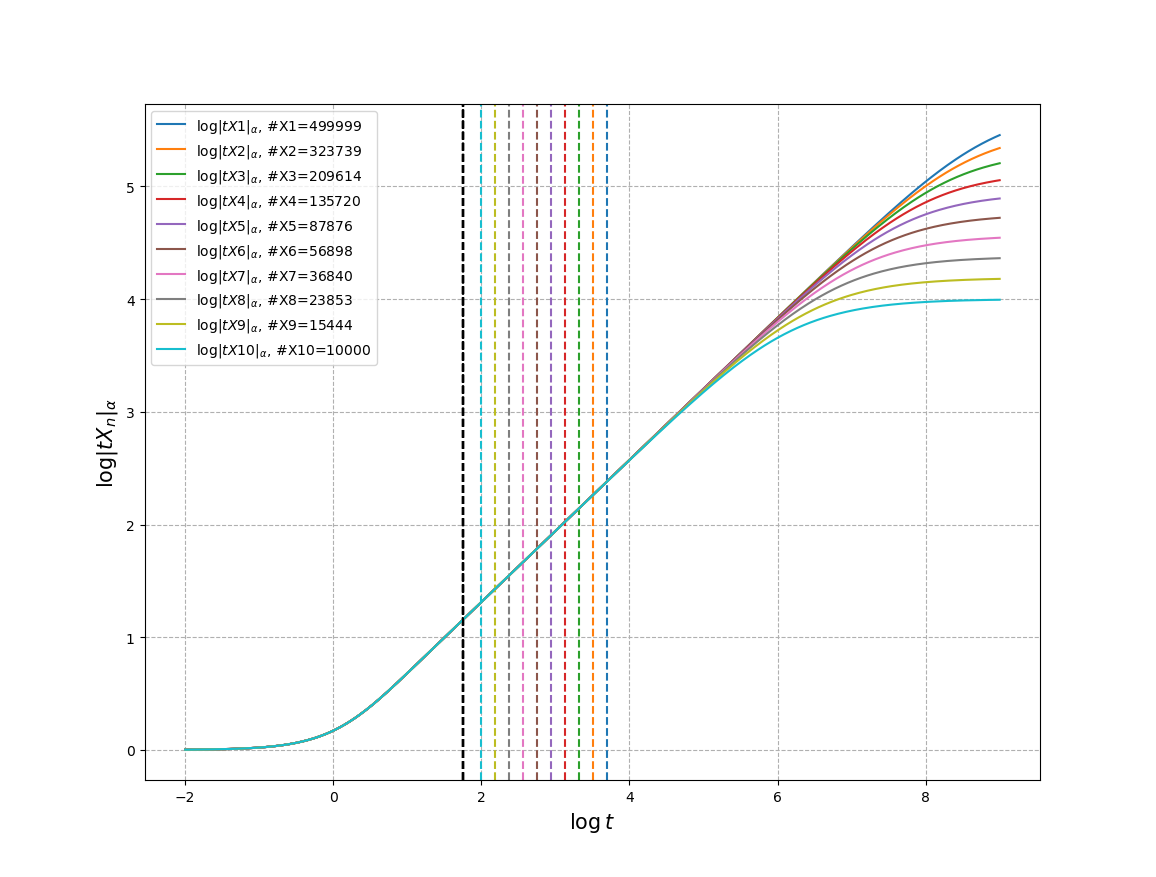} & &
     \includegraphics[scale=.31]{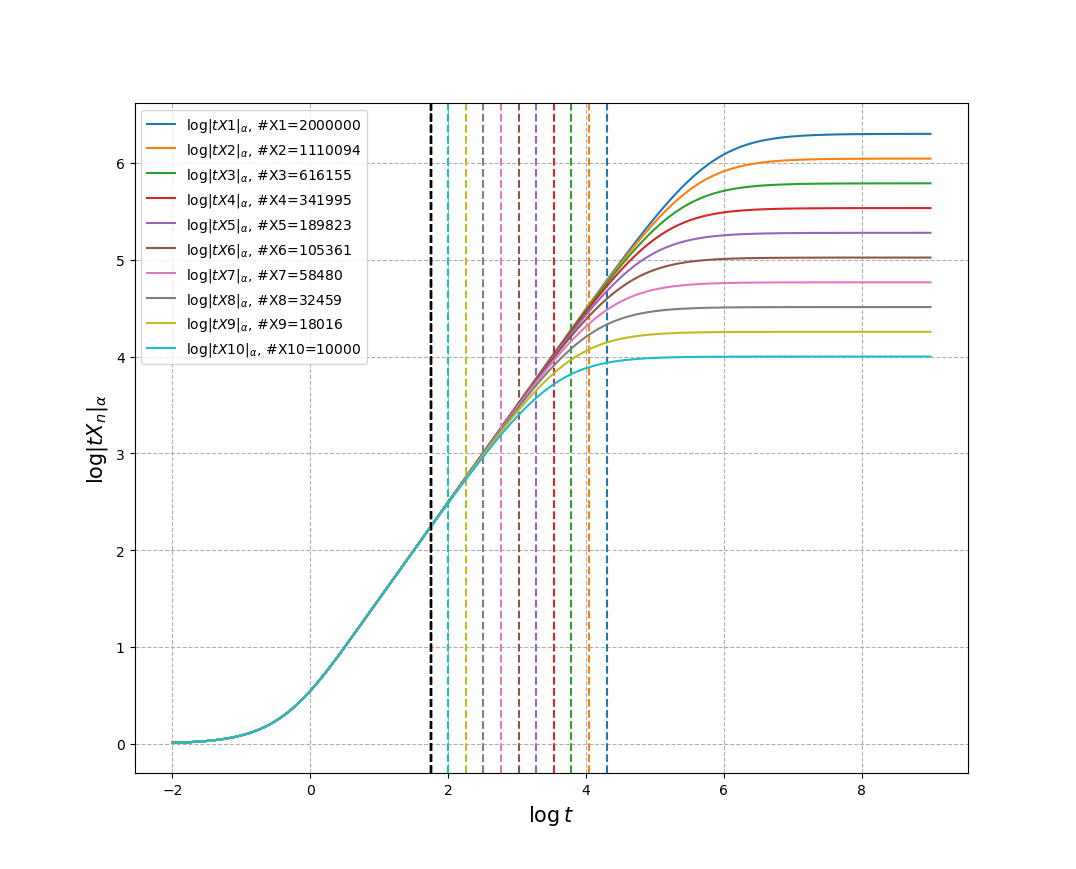}\\
     \notag &(a) && (b)
     \end{alignat}
    \caption{ (a)
    Multiple regression plots for samples of points of size varying from $10^4$ to $5*10^6-1$ taken from the Cantor set. We note that the sizes were chosen so as to be spaced logarithmically. We plot the different ranges over which we estimate the power law with dashed lines: the black dotted lines correspond to the left endpoint, which is the same for all regression plots, while the right endpoints are indicated by dashed lines in the same colour as the colour used to plot the alpha magnitude for a specific sample. 
    For instance, if we consider the dark blue dashed line,  this corresponds to the right endpoint of the range over which we estimate the power law for the alpha magnitude of the sample of $5*10^6-1$ points. 
    We note that additional points produce a slightly more accurate estimate, and cause the interval within which the plot follows a power rule to become wider. We could justifiably take the bounds to be larger in this case, but the use of more restrictive bounds is necessary in other cases and doesn't substantially affect the estimate here, so we constrain these for the purpose of consistency.
    (b) Multiple log-log plots for alpha magnitudes of samples from the unit circle. We again observe extended convergence arising from additional samples. The bounds need to be chosen more restrictively in this case, but additional points still seem to extend the interval of convergence in a linear fashion under logarithmic scale.
    }
    \label{fig:cant-circ-multiregplot}

\end{figure}

\subsection{Circle}
We next consider the unit circle $\SSS^1\subset \mathbb{R}^2$. 
In Proposition \ref{circlealpha} we compute the alpha magnitude of $\SSS^1$, and in Example \ref{E: amd circle} its alpha magnitude dimension. Here we sample $2* 10^6$ points uniformly at random from the circle. We then proceed to estimate the alpha magnitude for $\SSS^1$ in the manner described above. Once again, our estimate is close to the true alpha magnitude dimension as illustrated from Figure \ref{E: amd circle} (a). We illustrate how the ranges over which we estimate the power law change as the size of the samples changes in Figure \ref{fig:cant-circ-multiregplot}(b), while in Figure \ref{cantor-circle} (d) we provide the estimates obtained for these ranges. We observe that the estimates become more accurate as the number of points increases.

\subsection{Feigenbaum Attractor}

The Feigenbaum attractor  is a subset $F\subset I\subset\mathbb{R}$ generated by the recursive function 
\begin{equation}\label{E:feigenbaum}
x_{n+1}=a x_n(1-x_n) 
\end{equation}
where $a\approx 3.56995$. The equation $x_{n+1}=a x_n(1-x_n)$ is called the ``logistic map'' or ``logistic difference equation'' and describes the growth of a population: for $0<x_n<1$ and $1<a < 4$ the equation describes non-trivial dynamical behaviour, namely, the growth of a population that does not become extinct. We illustrate the attractors of the logistic map for different values of the parameter $a$ in Figure \ref{feigenbaum}(a).
 The   value of $a$ corresponding  to the Feigenbaum  attractor is  the point at which the `` `chaotic' region begins'' \cite{RM76}.

For the Feigenbaum attractor, the values taken by the different measures of fractal dimensions are only known approximately, since it is not known how to calculate them precisely.  The Hausdorff dimension of the Feigenbaum attractor has been computed to be approximately $0.538$ \cite{PG81}, the correlation dimension to be $0.5\pm.005$, and the information dimension to be $0.517$ \cite{PG83}. Here we estimate the alpha magnitude dimension of the Feigenbaum attractor to be approximately $0.497$.

\begin{figure}[t]
    \begin{alignat}{2}
&\notag \includegraphics[scale=.4]{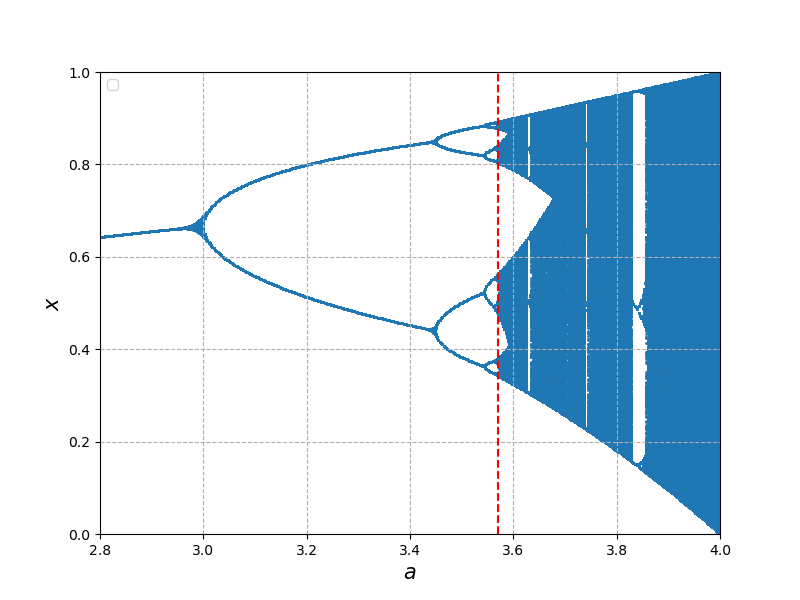}
&&\notag  \includegraphics[scale=.4]{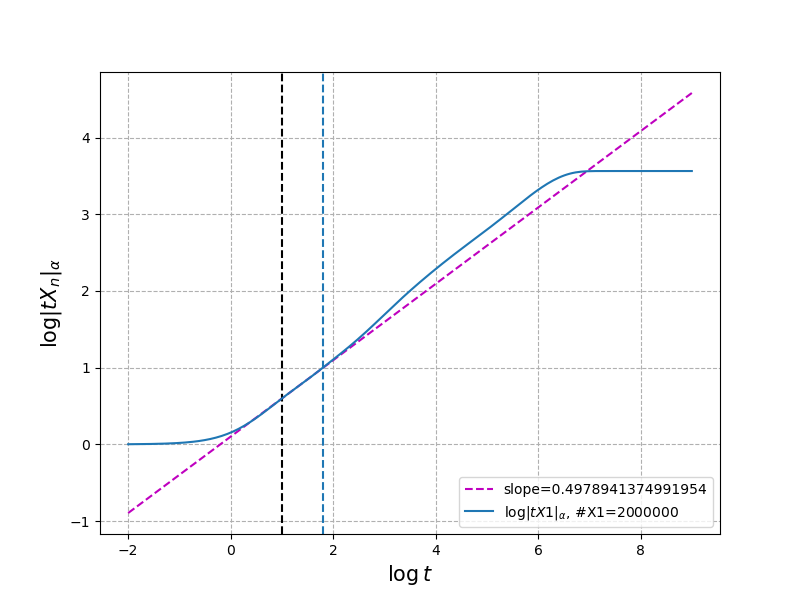} \\
    &\notag\qquad \qquad \qquad \qquad\qquad \qquad(a)
    &&\notag \qquad \qquad\qquad \qquad\qquad \qquad(b)
\end{alignat}
     \caption{(a)  Plot of attractors of the logistic map for different values of $a$ in Equation \ref{E:feigenbaum}. The Feigenbaum attractor is the vertical slice of points occurring at $a=3.56995$, marked in red. (b) Log-log plot for points sampled from the Feigenbaum attractor. For this estimate, the appropriate bounds for convergence were shifted slightly. For the choice of range we note that our experiments show that the datasets are more informative for lower values of $t$, in other words, when the points are closer to each other. 
     }
     \label{feigenbaum}
\end{figure}

\section{Conclusion}
Motivated by the study of point-cloud data, in our work we introduce a
new invariant of metric spaces, called alpha magnitude. This rests on the
definition of persistent magnitude introduced by Govc and Hepworth, and has
the advantage that it can be easily computed for medium-sized data sets.
We compute alpha magnitude for several compact subsets of Euclidean space, and conjecture that it encodes the Minkowski dimension of compact subsets of Euclidean space, whenever it is defined.  
Motivated by this conjecture, as well as by the computational advantages of alpha magnitude over magnitude, we propose that alpha magnitude can be used to estimate the fractal dimension of  datasets exhibiting self-similar properties, which is a direction that we plan to pursue in future work. 

\section{Data sets and code availability}

The data sets and code for the computations performed in this paper can be found at \url{https://github.com/miguelomalley/alpha-magnitude}.

\bibliography{EucMagContPrf}

\newcommand{\etalchar}[1]{$^{#1}$}
\begin{thebibliography}{CdSMS16}

\bibitem[BDKF20]{bunch2020weighting}
Eric Bunch, Daniel Dickinson, Jeffery Kline, and Glenn Fung.
\newblock Weighting vectors for machine learning: numerical harmonic analysis
  applied to boundary detection.
\newblock In {\em NeurIPS 2020 Workshop on Topological Data Analysis and
  Beyond}, 2020.

\bibitem[BP17]{CB17}
Christopher~J. Bishop and Yuval Peres.
\newblock {\em Fractals in Probability and Analysis}.
\newblock Cambridge University Press, 2017.

\bibitem[Car14]{carlsson2014}
Gunnar Carlsson.
\newblock Topological pattern recognition for point cloud data.
\newblock {\em Acta Numerica}, 23:289--368, 2014.

\bibitem[CdSMS16]{CDGO16}
F.~Chazal, V.~de~Silva, Glisse M., and Oudot S.
\newblock {\em The Structure and Stability of Persistence Modules}.
\newblock Springer, 2016.

\bibitem[EH10]{EH10}
H.~Edelsbrunner and J.~Harer.
\newblock {\em Computational Topology: An Introduction}.
\newblock Applied Mathematics. American Mathematical Society, 2010.

\bibitem[GH21]{GH21}
D.~Govc and R.~Hepworth.
\newblock Persistent magnitude.
\newblock {\em Journal of Pure and Applied Algebra}, 225, 2021.

\bibitem[Gra81]{PG81}
P.~Grassberger.
\newblock On the hausdorff dimension of fractal attractors.
\newblock {\em Journal of Statistical Physics}, 1981.

\bibitem[Gra83]{PG83}
P.~Grassberger.
\newblock Generalized dimensions of strange attractors.
\newblock {\em Physics Letters}, 1983.

\bibitem[GSP12]{gneiting}
T.~Gneiting, H.~Sevcikova, and D.B. Percival.
\newblock {Estimators of Fractal Dimension: Assessing the Roughness of Time
  Series and Spatial Data}.
\newblock {\em Statistical Science}, 27(2):247 -- 277, 2012.

\bibitem[HW17]{HW17}
R.~Hepworth and S.~Willerton.
\newblock Categorifying the magnitude of a graph.
\newblock {\em Homology, Homotopy and Applications}, 19(2):31--60, 2017.

\bibitem[JS20]{JS20}
J.~Jaquette and B.~Schweinhart.
\newblock Fractal dimension estimation with persistent homology: A comparative
  study.
\newblock {\em Communications in Nonlinear Science and Numerical Simulation},
  84, 2020.

\bibitem[Lay07]{lay2007convex}
S.R. Lay.
\newblock {\em Convex Sets and Their Applications}.
\newblock Dover Books on Mathematics Series. Dover Publications, 2007.

\bibitem[LC12]{LC12}
T.~Leinster and C.~A. Cobbold.
\newblock Measuring diversity: the importance of species similarity.
\newblock {\em Ecology}, 93:477--489, 2012.

\bibitem[Lei13]{leinster2010}
Tom Leinster.
\newblock The magnitude of metric spaces.
\newblock {\em Documenta Mathematica}, 18:857--905, 2013.

\bibitem[Lei21]{leinster21}
T.~Leinster.
\newblock {\em Entropy and diversity}.
\newblock Cambridge University Press, 2021.

\bibitem[LM16]{LM16}
T.~Leinster and M.~W. Meckes.
\newblock Maximizing diversity in biology and beyond.
\newblock {\em Entropy}, 18(88), 2016.

\bibitem[LM17]{LM17}
T.~Leinster and M.~Meckes.
\newblock The magnitude of a metric space: from category theory to geometric
  measure theory.
\newblock {\em Measure Theory in Non-Smooth Spaces}, 2017.

\bibitem[LR21]{LR21}
T.~Leinster and E.~Roff.
\newblock The maximum entropy of a metric space.
\newblock {\em Quarterly Journal of Mathematics}, 72:1271--1309, 2021.

\bibitem[LS21]{LS17}
T.~{Leinster} and M.~{Shulman}.
\newblock {Magnitude homology of enriched categories and metric spaces}.
\newblock {\em Algebraic $\&$ Geometric Topology}, pages 2175--2221, 2021.

\bibitem[Mec15]{Meckes2015}
M.W. Meckes.
\newblock Magnitude, diversity, capacities, and dimensions of metric spaces.
\newblock {\em Potential Analysis}, 42:549–572, 2015.

\bibitem[MM76]{RM76}
R.~M.~May.
\newblock Simple mathematical models with very complicated dynamics.
\newblock {\em Nature}, 261:459--467, 1976.

\bibitem[OPT{\etalchar{+}}17]{roadmap}
N.~Otter, M.~A. Porter, U.~Tillmann, P.~Grindrod, and H.~A. Harrington.
\newblock A roadmap for the computation of persistent homology.
\newblock {\em EPJ Data Science}, 6, 2017.

\bibitem[{Ott}18]{O18}
Nina {Otter}.
\newblock {Magnitude meets persistence. Homology theories for filtered
  simplicial sets}.
\newblock {\em arXiv e-prints}, page arXiv:1807.01540, July 2018.
\newblock To appear in {\em Homology, Homotopy and Applications}.

\bibitem[Raz02]{RR02}
Ran Raz.
\newblock On the complexity of matrix product.
\newblock In {\em Proceedings of the Thiry-Fourth Annual ACM Symposium on
  Theory of Computing}, STOC '02, page 144–151, New York, NY, USA, 2002.
  Association for Computing Machinery.

\bibitem[Ste09]{JS09}
J.~Stewart.
\newblock {\em Multivariable Calculus (4th ed.)}.
\newblock Cengage Learning, 2009.

\bibitem[{Wil}09]{W09}
Simon {Willerton}.
\newblock {Heuristic and computer calculations for the magnitude of metric
  spaces}.
\newblock {\em arXiv e-prints}, page arXiv:0910.5500, October 2009.

\end{thebibliography}
\bibliographystyle{alpha}

\end{document}